\documentclass[a4paper,10pt]{amsart}
\usepackage{amssymb}
\usepackage{amsthm}
\usepackage{amsmath}
\usepackage{amscd}
\usepackage{enumerate}
\usepackage{algorithmic}
\usepackage{algorithm}
\usepackage{caption}
\usepackage[usenames,dvipsnames]{pstricks}
\usepackage{epsfig}
\usepackage{pst-grad} 
\usepackage{pst-plot} 
\usepackage{hyperref}

\newtheorem{thm}{Theorem}
\newtheorem{prop}[thm]{Proposition}
\newtheorem{lem}[thm]{Lemma}
\theoremstyle{definition}
\newtheorem{defi}[thm]{Definition}

\newtheorem{rmk}[thm]{Remark}
\newtheorem{rmks}[thm]{Remarks}

\newcommand{\Z}{\mathbb{Z}}

\newcommand{\R}{\mathbb{R}}
\newcommand{\Q}{\mathbb{Q}}
\newcommand{\C}{\mathbb{C}}
\newcommand{\Proj}{\mathbb{P}}
\newcommand{\proba}{\mathcal{P}}
\newcommand{\Hamil}{\mathbb{H}}
\newcommand{\order}{\mathcal{O}}
\newcommand{\hyp}{\mathcal{H}}
\newcommand{\hypb}{\mathcal{B}}

\newcommand{\act}{\cdot}
\newcommand{\fdom}{\mathcal{F}}
\newcommand{\quatalg}[2]{\left(\frac{#1}{#2}\right)}
\newcommand{\prm}{\mathfrak{p}}
\newcommand{\Loba}{\mathcal{L}}
\newcommand{\gp}[1]{\langle #1 \rangle}
\newcommand{\sm}[1]{\left(\begin{smallmatrix}#1\end{smallmatrix}\right)}
\newcommand{\tg}{\tilde{g}}
\newcommand{\tw}{\tilde{w}}

\newcommand{\tA}{\tilde{A}}
\newcommand{\tB}{\tilde{B}}
\newcommand{\tC}{\tilde{C}}
\newcommand{\tD}{\tilde{D}}
\newcommand{\sd}{\sqrt{2}}
\newcommand{\rs}{\sqrt{6}}

\newcommand{\nmg}{\|g\|}

\DeclareMathOperator{\SL}{SL}

\DeclareMathOperator{\PSL}{PSL}
\DeclareMathOperator{\PGL}{PGL}
\DeclareMathOperator{\SU}{SU}
\DeclareMathOperator{\PSU}{PSU}
\DeclareMathOperator{\GL}{GL}

\DeclareMathOperator{\mat}{\mathcal{M}}
\DeclareMathOperator{\trd}{trd}
\DeclareMathOperator{\nrd}{nrd}

\DeclareMathOperator{\dist}{d}

\DeclareMathOperator{\diff}{d\!}
\DeclareMathOperator{\vol}{Vol}
\DeclareMathOperator{\covol}{Covol}
\DeclareMathOperator{\charac}{char}
\DeclareMathOperator{\Ext}{Ext}
\DeclareMathOperator{\Int}{Int}
\DeclareMathOperator{\I}{I}
\DeclareMathOperator{\reduc}{Red}

\DeclareMathOperator{\rad}{rad}
\DeclareMathOperator{\invrad}{invrad}
\DeclareMathOperator{\tr}{tr}
\DeclareMathOperator{\ball}{B}

\title{Computing arithmetic Kleinian groups}
\author{Aurel Page}\thanks{
Aurel Page\\
Univ. Bordeaux, IMB, UMR 5251, F-33400 Talence, France.\\
CNRS, IMB, UMR 5251, F-33400 Talence, France.\\
INRIA, F-33400 Talence, France.\\
\texttt{aurel.page@math.u-bordeaux1.fr}}
\date{}

\begin{document}

\begin{abstract}
Arithmetic Kleinian groups are arithmetic lattices in~$\PSL_2(\C)$. We present an algorithm that, given such a group~$\Gamma$, returns a fundamental domain and a finite presentation for~$\Gamma$ with a computable isomorphism.
\end{abstract}

\maketitle

\tableofcontents
\newpage

\section*{Introduction}

An arithmetic Kleinian group~$\Gamma$ is a discrete subgroup of~$\PGL_2(\C)$ with finite covolume that is commensurable with the image of the integral points of a form of~$\GL_2$ defined over a number field under the natural surjection~$\GL_2(\C)\to\PGL_2(\C)$.
Our main results are new deterministic and probabilistic algorithms for constructing fundamental domains for the action of arithmetic Kleinian groups~$\Gamma$ on hyperbolic three-space that produce a finite presentation for~$\Gamma$. There is a substantial literature concerning such algorithms, some of which we review below. We compare our algorithms to recent ones and discuss numerical evidence suggesting that ours are more efficient. The algorithm presented here prepares the ground for computing the cohomology of these groups with the action of Hecke operators, which gives a concrete realization of certain automorphic forms by the Matsushima-Murakami formula \cite{borelwallach}. By the Jacquet-Langlands correspondence \cite{jacquet1972automorphic}, such forms are essentially the same as automorphic forms for~$\GL_2$ over some number field. They should have attached Galois representations, but the construction of these representations in general is still an open problem.
More generally, the integral (co)homology of such groups has recently received a lot of attention: for example the size of their torsion \cite{bvtorsion} and arithmetic functoriality \cite{venkcaltorsion} are being actively studied. Our algorithm allows for empirical study of these objects.

%
%
\par The problem of computing fundamental domains for such groups is well studied. In the analogous Fuchsian group case, i.e. a subgroup of~$\PSL_2(\R)$, an algorithm may have been known to Klein and J. Voight \cite{voightfuchsian} has described and implemented an efficient algorithm exploiting reduction theory. In the special case of Bianchi groups, i.e. when the base field is imaginary quadratic and the group is split, R.G.~Swan \cite{swan} has described an algorithm, which was implemented by Riley~\cite{riley} and A. Rahm \cite{rahmhomologies}; D.~Yasaki \cite{yasaki2010hyperbolic} has described and implemented another algorithm based on Vorono\" i theory. C.~Corrales, E.~Jespers, G.~Leal and \'A.~del~R\'io \cite{Corrales} have described an algorithm for the general Kleinian group case. They implemented it for one nonsplit group with imaginary quadratic base field. Our algorithm and implementation are more general, and experimentally more efficient in practice.
We have recently found an unpublished algorithm of K.~N.~Jones and A.~W.~Reid, mentioned and briefly described in~\cite[section~3.1]{chinburg2001arithmetic} that solves the same problem.
\par The article is organized as follows. In the first section we recall basic definitions and properties of hyperbolic geometry, quaternion algebras and Kleinian groups. In the second section we describe our algorithms: basic procedures to work in the hyperbolic $3$-space, algorithms for computing a Dirichlet domain and a presentation with a computable isomorphism for a cocompact Kleinian group, and how to apply these algorithms to arithmetic Kleinian groups. In the third section we show examples produced by our implementation of these algorithms and comment on their running time. 

\bigskip

I would like to thank John Voight for proposing me this project and supervising my master thesis, and Karim Belabas and Andreas Enge for their helpful comments on earlier versions of this article. Experiments presented in this paper were carried out using the PLAFRIM experimental testbed, being developed under the Inria PlaFRIM development action with support from LABRI and IMB and other entities: Conseil R\'egional d'Aquitaine, FeDER, Universit\'e de Bordeaux and CNRS (see \url{https://plafrim.bordeaux.inria.fr/}). This research was partially funded by ERC Starting Grant ANTICS 278537.

\section{Arithmetic Kleinian groups}
Here we recall basic definitions and properties of hyperbolic geometry, quaternion algebras and Kleinian groups. The general reference for this section is \cite{mac}.
\subsection{Hyperbolic geometry}
The reader can find more about hyperbolic geometry in \cite{foundations}. The \emph{upper half-space} is the Riemannian manifold~$\hyp^3 = \C\times\R_{>0}$ with Riemannian metric given by 
\[\diff s^2 = \frac{\diff x^2+\diff y^2+\diff t^2}{t^2}\]
where $(z,t)\in\hyp^3$, $z=x+iy$ and $t>0$.
For $w,w'\in\hyp^3$, $\dist(w,w')$ is the distance between $w$ and $w'$.
The set $\Proj^1(\C)$ is called the \emph{sphere at infinity}. 
The upper half-space is a model of the hyperbolic~$3$-space, i.e. the unique connected, simply connected Riemannian manifold  with constant sectional curvature~$-1$. In this space, the volume of the ball of radius~$r$ is~$ \pi(\sinh(2r)-2r)$.
\par The group~$\PSL_2(\C)$ acts on~$\hyp^3$ in the following way. Consider the ring of Hamiltonians~$\Hamil = \C + \C j$ with multiplication given by~$j^2=-1$ and~$jz=\bar{z}j$ for~$z\in\C$, and identify~$\hyp^3$ with the subset~$\C+\R_{>0}j\subset\Hamil$. Then for an element~$g=\sm{a & b\\ c & d}\in\SL_2(\C)$ and $w\in\hyp^3$, the formula
\begin{equation*}g\act w = (aw+b)(cw+d)^{-1}=(wc+d)^{-1}(wa+b)\label{eqh3}\end{equation*}
defines an action of~$\PSL_2(\C)$ on~$\hyp^3$ by orientation-preserving isometries. This action is transitive and the stabilizer of the point~$j\in\hyp^ 3$ in $\PSL_2(\C)$ is the subgroup~$\PSU_2(\C)$.
\par The trace of an element of~$\PSL_2(\C)$ is defined up to sign, and we have the following classification of conjugacy classes in $\PSL_2(\C)$:
\begin{itemize}
	\item If $\tr(g)\in\C\setminus[-2,2]$, then $g$ has two distinct fixed points in $\Proj^1(\C)$, no fixed point in~$\hyp^3$ and stabilizes the geodesic between its fixed points, called its \emph{axis}. The element~$g$ is conjugate to $\pm\sm{\lambda & 0 \\ 0 & \lambda^{-1}}$  with $|\lambda|>1$; it is called \emph{loxodromic}.
	\item If $\tr(g)\in (-2,2)$, then $g$ has two distinct fixed points in $\Proj^1(\C)$, and fixes every point in the geodesic between these two fixed points. The element~$g$ is conjugate to~$\pm\sm{e^{i\theta} & 0 \\ 0 & e^{-i\theta}}$ with $\theta\in\R\setminus(\pi+2\pi\Z$); it is called \emph{elliptic}.
	\item If $\tr(g) = \pm 2$, then $g$ has one fixed point in $\Proj^1(\C)$ and no fixed point in~$\hyp^3$. It is conjugate to $\pm\sm{1 & 1 \\ 0 & 1}$; it is called \emph{parabolic}.
\end{itemize}

\subsection{The unit ball model}

In actual computations we are going to work with another model of the hyperbolic $3$-space. The \emph{unit ball} $\hypb$ is the open ball of center~$0$ and radius $1$ in~$\R^3\cong \C+\R j\subset\Hamil$, equipped with the Riemannian metric
\[\diff s^2 = \frac{4(\diff x^2+\diff y^2+\diff t^2)}{(1-|w|^2)^2}\]
where $w=(z,t)\in\hypb$, $z=x+iy$ and $|w|^2=x^2+y^2+t^2< 1$. The \emph{sphere at infinity}~$\partial\hypb$ is the Euclidean sphere of center $0$ and radius~$1$. 
The distance between two points~$w,w'\in\hypb$ is given by the explicit formula
\[\dist(w,w')=\cosh^{-1}\left(1+2\dfrac{|w-w'|^2}{(1-|w|^2)(1-|w'|^2)}\right)\text{.}\]
The upper half-space and the unit ball are isometric, the isometry being given by
\begin{equation*}
\eta \colon
\left\{
	\begin{aligned}
		\hyp^3 	& \longrightarrow \hypb\\
		w 		& \longmapsto (w-j)(1-jw)^{-1}=(1-wj)^{-1}(w-j)\text{,}
	\end{aligned}
\right.
\end{equation*}
and the corresponding action of an element~$g=\sm{a & b\\ c & d}\in\PSL_2(\C)$ on a point~$w\in\hypb$ is given by~\begin{equation}\label{actionformula}g\act w = (Aw+B)(Cw+D)^{-1}\end{equation}
	where
	\begin{equation*}A=a+\bar{d}+(b-\bar{c})j,\ B=b+\bar{c}+(a-\bar{d})j,\end{equation*}
	\begin{equation*}C=c+\bar{b}+(d-\bar{a})j,\ D=d+\bar{a}+(c-\bar{b})j\text{.}\end{equation*}

\medskip


In the unit ball model, the geodesic planes are the intersections with~$\hypb$ of Euclidean spheres and Euclidean planes orthogonal to the sphere at infinity, and the geodesics are the intersections with~$\hypb$ of Euclidean circles and Euclidean straight lines orthogonal to the sphere at infinity. A \emph{half-space} is an open connected subset of~$\hypb$ with boundary consisting of a geodesic plane. A \emph{convex polyhedron} is the intersection of a set of half-spaces, such that the corresponding set of geodesic planes is locally finite.

\subsection{The Lobachevsky function and volumes of tetrahedra}

We are going to compute hyperbolic volumes, and for this the main tool is going to be the Lobachevsky function, which we define here. The integral
\[-\int_0^{\theta}\ln |2\sin u|\diff u\]
converges for $\theta\in\R\setminus\pi\Z$ and admits a continuous extension to $\R$ that is odd and periodic with period~$\pi$. This extension is called the \emph{Lobachevsky function} $\Loba(\theta)$. The Lobachevsky function admits a power series expansion, converging for~$\theta\in [-\pi,\pi]$:
\[\Loba(\theta)=\theta\left(1-\ln(2|\theta|)+\sum_{n=1}^\infty\frac{\zeta(2n)}{n(2n+1)}\Bigl(\frac{\theta}{\pi}\Bigr)^{2n}\right)\text{.}\]
With this function we can derive a formula for the volume of a certain standard tetrahedron. We will use it to compute the volume of convex polyhedra.

\begin{prop}\label{propvol}Let $T$ be the tetrahedron in $\hyp^3$ with one vertex at $\infty$ and the other vertices $A,B,C$ on the unit hemisphere projecting vertically onto~$A',B',C'$ in $\C$ with $A'=0$ to form a Euclidean triangle, with angles $\frac{\pi}{2}$ at~$B'$ and~$\alpha$ at $A'$, and such that the angle along $BC$ is $\gamma$. Then the volume of~$T$ is finite and given by
\[\vol(T)=\frac14\left[ \Loba(\alpha+\gamma) + \Loba(\alpha-\gamma) + 2\Loba\left(\frac{\pi}{2}-\alpha\right) \right]\text{.}\]
\end{prop}
\begin{proof}
 This formula can be found in \cite[paragraph 1.7]{mac}.
\end{proof}

\subsection{Kleinian groups, Dirichlet domains and exterior domains}\label{sectiondirdom}

A subgroup~$\Gamma$ of~$\PSL_2(\C)$ is a \emph{Kleinian group} if it acts discontinuously on~$\hyp^3$, or equivalently if it is a discrete subgroup of~$\PSL_2(\C)$. A \emph{fundamental domain} for $\Gamma$ is an open subset $\fdom$ of $\hyp^3$ such that
\begin{enumerate}[(i)]
	\item $\bigcup_{\gamma\in\Gamma}\gamma\overline{\fdom} = \hyp^3$;
	\item For all $\gamma\in\Gamma\setminus\{1\},\ \fdom\cap\gamma\fdom = \emptyset$;
	\item\label{condvol} $\vol(\partial\fdom) = 0$
\end{enumerate}
where~$\vol$ is the Riemannian volume on~$\hyp^3$.
To compute a fundamental domain for a Kleinian group~$\Gamma$, we are going to use the standard construction of Dirichlet domains. 
Let $p\in\hypb$ be a point with trivial stabilizer in $\Gamma$. Then the \emph{Dirichlet domain} centered at~$p$
\[D_p(\Gamma) = \{x\in\hypb\ | \text{ for all } \gamma\in\Gamma\setminus\{1\},\ \dist(x,p) < \dist(\gamma x,p)\}\]
is a convex fundamental polyhedron for $\Gamma$. If~$\Gamma$ has finite covolume, then the closure of~$D_p(\Gamma)$ has finitely many faces. A Kleinian group~$\Gamma$ is \emph{geometrically finite} if the closure of one (equivalently, every) Dirichlet domain for~$\Gamma$ has finitely many faces.

Note that since~$\Gamma$ acts properly discontinuously on~$\hypb$, every point outside a zero measure, closed subset of~$\hypb$ has a trivial stabilizer in~$\Gamma$. In the unit ball model, the Dirichlet domain centered at~$0$ has a simple description. Consider an element~$g\in\SL_2(\C)$ not fixing $0\in\hypb$. Let
\begin{itemize}
	\item $\I(g)=\{w\in\hypb\ |\ \dist(w,0)=\dist(g w,0)\}$;
	\item $\Ext(g)=\{w\in\hypb\ |\ \dist(w,0)<\dist(g w,0)\}$;
	\item $\Int(g)=\{w\in\hypb\ |\ \dist(w,0)>\dist(g w,0)\}$.
\end{itemize}
We call $\I(g)$ the \emph{isometric sphere} of $g$. For a subset~$S\subset\SL_2(\C)$ such that no element of $S$ fixes $0$, the \emph{exterior domain} of $S$ is~$\Ext(S)=\bigcap_{g\in S}\Ext(g)$. The set~$S$ is a \emph{defining set} for $\Ext(S)$. A \emph{minimal defining set} for $\Ext(S)$ is a subset~$S'\subset S$ such that $\Ext(S')=\Ext(S)$ and for all $g\in S'$, the geodesic plane~$\I(g)$ contains a face of~$\overline{\Ext(S)}$.

\medskip

With these definitions it is clear that~$D_0(\Gamma) = \Ext(\Gamma\setminus\{1\})$. Note that for all~$p\in\hypb$ with trivial stabilizer in~$\Gamma$,~$D_p(\Gamma)=u D_0(u^{-1}\Gamma u)$ where~$u\in\PSL_2(\C)$ is such that~$p = u\act 0$, so there is no harm in restricting to the Dirichlet domain centered at~$0$. Consider an element~$g\in\SL_2(\C)$ and $A,B,C,D$ as in formula~\eqref{actionformula}. Then $g\act 0=0$ if and only if $C=0$ and, if $g$ does not fix $0$, then a simple but lengthy computation reveals that~$\I(g)$ is the intersection of~$\hypb$ and the Euclidean sphere of center~$w$ and radius~$r$, where 
\begin{equation}\label{eqradius}
 w = -C^{-1}D\text{ and }r = 2/|C|\text{,}
\end{equation}
and that~$\Int(g)$ is the interior of this sphere. The details are in~\cite[Proposition 3.1.6]{monmemoire}.

\medskip

Another property of Dirichlet domains is their rich structure: it gives a presentation for the group, and also necessary and sufficient conditions for an exterior domain to be a fundamental fomain. Suppose~$\Gamma$ is a Kleinian group in which~$0$ has trivial stabilizer, and let~$g,h\in\Gamma$. Then we have~$\I(g)=\I(h)$ if and only if~$g=h$. We also have~$g\I(g)=\I(g^{-1})$, and a point~$x\in\I(g)$ is in the defining set of $D_0(\Gamma)$ if and only if $gx\in\I(g^{-1})$ is too.
\par From this, we can group the faces of~$\overline{D_0(\Gamma)}$ in pairs, one contained in some~$\I(g)$ and the other contained in~$\I(g^{-1})$, and~$g,g^{-1}$ send the faces to each other. This is the~\emph{face pairing} structure, and the elements~$g$ such that~$\I(g)$ contains a face of~$\overline{D_0(\Gamma)}$ are called the~\emph{face pairing transformations}. They generate the group~$\Gamma$.
\par Now we are going to look for relations. The first type comes from edge cycles: consider an edge~$e_1$ of~$\overline{D_0(\Gamma)}$ contained in some~$\I(g)\cap \I(h)$, and let~$g_1=g$. We define inductively a sequence of edges and elements in~$\Gamma$ in the following way. We let~$e_{n+1} = g_n e_n$. Then~$e_{n+1}$ is contained in~$\I(g_n^{-1})\cap\I(g_{n+1})$ for a unique~$\I(g_{n+1})$ (see Figure~\ref{cycle2d}). If~$\overline{D_0(\Gamma)}$ has finitely many faces, then the sequence~$(e_n,g_n)_n$ is periodic, let~$m$ be its period. The sequence of edges~$C=(e_1,\dots,e_m)$ is a~\emph{cycle} of edges, and~$m$ is its~\emph{length}. The~\emph{cycle transformation} at~$e_1$ is~$h=g_mg_{m-1}\dots g_1$, and has the property:
\begin{enumerate}[(i)]
\item The cycle transformation at~$e_1$ fixes~$e_1$ pointwise.
\end{enumerate}
This implies that~$h$ satisfies the~\emph{cycle relation}~$h^\nu=1$ for some integer~$\nu$. If~$\nu\neq 1$, the cycle is called~\emph{elliptic}. At every edge~$e_i$, the geodesic planes~$\I(g_i^{-1})$ and~$\I(g_{i+1})$ make an angle~$\alpha(e_i)$ inside~$D_0(\Gamma)$.
The~\emph{cycle angle} of~$C$ is~$\alpha(C)=\sum_{i=1}^m \alpha(e_i)$. Since the translates of~$D_0(\Gamma)$ cover a neighborhood of~$e_1$, we have the property:
\begin{enumerate}[(i)]\setcounter{enumi}{1}
  \item The cycle angle is~$\frac{2\pi}{\nu}$ where~$\nu$ is the order of the cycle transformation.
\end{enumerate}

\par The second type of relations comes from elements of order~$2$: it may happen that~$\I(g)=\I(g^{-1})$, then the element~$g$ satisfies the~\emph{reflection relation}~$g^2=1$.

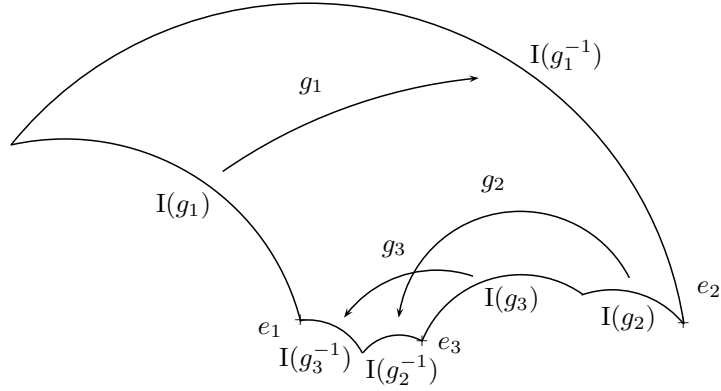
\begin{figure}[H]
\centering
\scalebox{1} 
{
\begin{pspicture}(1.5,2.55)(13.385,7.52)
\psarc[linewidth=0.02](6.4,2.51){5.0}{8.5752325}{141.3402}
\psarc[linewidth=0.02](3.2,2.51){3.2}{14.036243}{102.83561}
\psarc[linewidth=0.02](10.4,2.51){1.2}{38.65981}{109.53665}
\psarc[linewidth=0.02](9.2,2.51){1.4}{54.29331}{158.49857}
\psarc[linewidth=0.02](7.6,2.51){0.6}{60.25512}{142.43141}
\psarc[linewidth=0.02](6.4,2.51){0.8}{26.565052}{97.76517}
\psdots[dotsize=0.12,dotstyle=+](11.34,3.27)
\psdots[dotsize=0.12,dotstyle=+](7.9,3.03)
\psdots[dotsize=0.12,dotstyle=+](6.3,3.31)
\psarc[linewidth=0.02,arrowsize=0.05291667cm 2.0,arrowlength=1.4,arrowinset=0.4]{<-}(9.31,-0.48){7.03}{95.41553}{125.0031}
\psarc[linewidth=0.02,arrowsize=0.05291667cm 2.0,arrowlength=1.4,arrowinset=0.4]{->}(9.19,3.14){1.61}{26.87814}{172.87498}
\psarc[linewidth=0.02,arrowsize=0.05291667cm 2.0,arrowlength=1.4,arrowinset=0.4]{->}(8.09,2.48){1.49}{71.21138}{144.75243}
\usefont{T1}{ptm}{m}{n}
\rput(5.90125,3.18){$e_1$}
\rput(11.675781,3.72){$e_2$}
\rput(8.268281,2.98){$e_3$}
\rput(6.442656,6.42){$g_1$}
\rput(8.837188,5.12){$g_2$}
\rput(7.5296874,4.26){$g_3$}
\rput(4.7610936,4.82){$\I(g_1)$}
\rput(9.800468,6.82){$\I(g_1^{-1})$}
\rput(10.615625,3.33){$\I(g_2)$}
\rput(7.64,2.64){$\I(g_2^{-1})$}
\rput(9.08,3.58){$\I(g_3)$}
\rput(6.5,2.8){$\I(g_3^{-1})$}
\end{pspicture} 
}
\caption{A length three cycle in a planar cut}\label{cycle2d}
\end{figure}

\begin{thm}[Poincar\'e]\label{presthm}Let $D=D_0(\Gamma)$ be the Dirichlet domain of a geometrically finite Kleinian group~$\Gamma$. Then the face pairing transformation generate the group~$\Gamma$, and the reflection relations together with the cycle relations form a complete set of relations for~$\Gamma$.
\end{thm}

\begin{rmk}
 In the presentation given by the theorem we consider only one element for each pair of face-pairing transformation~$g,g^{-1}$. If we take both in the set of generators, we have to add the ``inverse'' relation~$g\, g^{-1} = 1$.
\end{rmk}

\par We are now looking for sufficient conditions for an exterior domain to be a fundamental domain. There is another necessary condition, coming from cycles of some special points at infinity. A point~$z\in\partial\hypb$ is a \emph{tangency vertex} if it is a point of tangency~$z=f\cap f'$ of two faces~$f\subset\I(g),f'\subset\I(g')$ of~$D_0(\Gamma)$. If~$z_1=\I(g_0)\cap\I(g_1)$ is a tangency vertex, then we define a sequence by letting~$z_{i+1}=g_i\act z_i=\I(g_i^{-1})\cap \I(g_{i+1})$ while~$z_{i+1}$ is a tangency vertex (otherwise the sequence ends at~$z_i$). If such a sequence~$(z_i)$ is infinite and~$D_0(\Gamma)$ has finitely many faces, then it is periodic. Let~$m$ be its period; then~$(z_1,\dots,z_m)$ is a \emph{tangency vertex cycle} and the \emph{tangency vertex transformation} is~$h=g_mg_{m-1}\dots g_1$. The fact that~$\hypb / \Gamma$ is complete implies the property:
\begin{enumerate}[(i)]\setcounter{enumi}{2}
  \item The tangency vertex transformation is parabolic.
\end{enumerate}

\medskip
Actually all these definitions can make sense for any exterior domain. Suppose~$\Ext(S)$ is an exterior domain with~$S\subset\Gamma$ a finite minimal defining set. We say that it has a face pairing if~$S=S^{-1}$ and for every~$g\in S$ the image by~$g$ of the face contained in~$\I(g)$ is the face contained in~$\I(g^{-1})$ -- equivalently, the image of every edge of~$\overline{\Ext(S)}$ by the pairing transformation of an adjacent face is an edge of~$\overline{\Ext(S)}$. This implies that every cycle is well-defined. We say that it satisfies the~\emph{cycle condition} if every cycle satisfies the properties (i) and~(ii), and that it is~\emph{complete} if every tangency vertex cycle satisfies the property~(iii).

\begin{thm}[Poincar\'e]\label{poincthm}Let $D=\Ext(S)$ be an exterior domain with~$S$ finite. Suppose~$D$ has a face pairing, satisfies the cycle condition, and is complete. Let $\Gamma'$ be the group generated by the face pairing transformations. Then $D$ is a fundamental polyhedron for $\Gamma'$.
\end{thm}

\begin{proof}
 Both theorems are a special case of the second Theorem in \cite{maskit1971poincar}.
\end{proof}

\subsection{Quaternion algebras and arithmetic Kleinian groups}
We can now describe the construction of arithmetic Kleinian groups using orders in quaternion algebras. The reader can find more about quaternion algebra in \cite{mfv}. A \emph{quaternion algebra}~$B$ over a field~$F$ is a central simple algebra of dimension~$4$ over~$F$. Equivalently, if~$\charac F\neq 2$, there exists~$a,b\in F^\times$ such that~$B=F+Fi+Fj+Fij$ with multiplication table given by~$i^2=a,\ j^2=b,\ ji=-ij$. Such an algebra is written~$B=\quatalg{a,b}{F}$. A quaternion algebra either is isomorphic to the matrix ring~$\mat_2(F)$, or is a division algebra. Given an element~$w=x+yi+zj+tij\in \quatalg{a,b}{F}$, we define its~\emph{conjugate}~$\bar{w}=x-yi-zj-tij$, its~\emph{reduced trace}~$\trd(w)=w+\bar{w}=2x\in F$ and its~\emph{reduced norm}~$\nrd(w)=w\bar{w}=x^2-ay^2-bz^2+abt^2\in F$.
\par Let~$F$ be a number field, let $\Z_F$ be its ring of integers and let~$B$ be a quaternion algebra over~$F$. An~\emph{order}~$\order\subset B$ is a finitely generated $\Z_F$-submodule with~$F\order = B$ that is also a subring. We write~$\order^1\subset\order^\times$ the subgroup of elements of reduced norm~$1$.
\par A place~$v$ of~$F$ is~\emph{split} or~\emph{ramified} depending on whether~$B\otimes_F F_v$ is isomorphic to the matrix ring or not. The set of ramified places is finite and the~\emph{discriminant} of~$B$ is the product of the ramified finite places, viewed as an ideal in~$\Z_F$. The number field~$F$ is \emph{almost totally real} (or \emph{ATR}) if it has exactly one complex place. A quaternion algebra over an ATR field is \emph{Kleinian} if it is ramified at every real place.

\begin{thm}
\par\noindent Let~$F$ be an ATR number field of degree~$n$, $B$ a Kleinian quaternion algebra over~$F$ and~$\order$ be an order in~$B$. Let~$\rho:B\hookrightarrow\mat_2(\C)$ be an algebra homomorphism extending a complex embedding of~$F$. Then the group~$\Gamma(\order) = \rho(\order^1)/\{\pm 1\}\subset\PSL_2(\C)$ is a Kleinian group. It has finite covolume, and it is cocompact if and only if~$B$ is a division algebra. Furthermore, if~$\order$ is maximal, we have
\begin{equation}\label{covolumeformula}\covol(\Gamma(\order)) = \dfrac{|\Delta_F|^{3/2}\zeta_{F}(2)\Phi(\Delta_B)}{(4\pi^2)^{n-1}}\end{equation}
where $\Delta_F$ is the discriminant of $F$, $\zeta_F$ is the Dedekind zeta function of $F$,~$\Delta_B$ is the discriminant of $B$ and $\Phi(\mathfrak{N})=N(\mathfrak{N})\cdot\prod_{\prm|\mathfrak{N}}\left(1-N(\prm)^{-1}\right)$ for every ideal~$\mathfrak{N}$ of~$F$.
\end{thm}

\begin{proof}
This theorem can be found in~\cite[Theorems 8.2.2, 8.2.3 and 11.1.3]{mac}. 
\end{proof}
 An \emph{arithmetic Kleinian group} is a Kleinian group that is commensurable with a group~$\Gamma(\order)$ as in the previous theorem. This is equivalent to the definition given in the introduction. The object of the next section is to describe an algorithm that, given such a group, computes a fundamental domain for~$\Gamma(\order)$, and a presentation with a computable isomorphism.

\section{Algorithms}
We describe every algorithm in ideal arithmetic. In section~\ref{float-impl}, we explain how to implement these algorithms using floating-point arithmetic.
\subsection{Algorithms for polyhedra in the hyperbolic $3$-space}
We start with low-level algorithms for dealing with hyperbolic polyhedra. A point in~$\hypb$ is represented by a vector in~$\C+\R j$; a geodesic plane not containing~$0$ is represented by the Euclidean center and radius of the corresponding Euclidean sphere; a geodesic not containing~$0$ is represented by the Euclidean center and radius of a Euclidean sphere and a basis of a Euclidean plane containing the center of the sphere, such that the geodesic is the intersection of~$\hypb$, this sphere and this plane.
\par Using these representations, it is an exercise in computational geometry to see that we can compute the faces, edges and vertices of a convex polyhedron given by a finite set of half-spaces containing~$0$. The details can be found in~\cite[section II.3.3]{monmemoire}. A harder task is to compute the volume of such a polyhedron. We describe an algorithm here; it is essentially the same as the one described in \cite[section 1.7]{mac} but for the sake of completeness we provide all the details here.
\par Algorithm \ref{algovol} computes the volume of a convex polyhedron with finitely many faces.

\begin{algorithm}[H]
\caption{Volume of a convex polyhedron}
\label{algovol}
\begin{algorithmic}[1]
	\REQUIRE A convex polyhedron $P$ with finitely many faces
	\ENSURE The hyperbolic volume of~$P$
	\STATE \label{triangles} Split every face of $P$ into triangles
	\STATE \label{steptetra} Split $P$ into tetrahedra
	\STATE Using the map~$\eta^{-1}$, send every tetrahedron back to~$\hyp^3$
	\STATE \label{infvertex}Express every tetrahedron as a difference of two tetrahedra, each having a vertex in the sphere at infinity
	\STATE For every tetrahedron having a vertex in the sphere at infinity, apply an isometry to map it to a tetrahedron with one vertex at $\infty$ and the other vertices on the unit hemisphere
	\STATE\label{jvertex} Express every such tetrahedron as a sum and difference of tetrahedra of the same type having one vertex at $j$
	\STATE Express every such tetrahedron as a sum and difference of tetrahedra of the same type with projected Euclidean triangle having a right angle not at~$0$
	\STATE\label{stepangles} For every such tetrahedron, compute the angles $\alpha$ and $\gamma$ and use Proposition~\ref{propvol} to compute the volume~\label{steploba}
	\STATE $\vol(P) \leftarrow$ sum of every contribution
	\RETURN $\vol(P)$
\end{algorithmic}
\end{algorithm}

\begin{figure}[ht]
\centering
\includegraphics*[width=6cm,keepaspectratio=true]{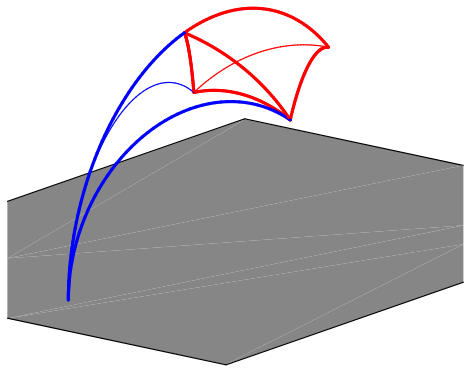}
\caption{Step \ref{infvertex} in Algorithm \ref{algovol}}\label{voltet}
\end{figure}

\begin{rmks}~
\begin{itemize}
	\item For step \ref{triangles}, choose a vertex of the face and link it to every other vertex;
	\item For step \ref{steptetra}, choose a vertex of $P$ and link it to every computed triangle;
	\item For step \ref{infvertex}, choose an edge and consider a geodesic ray containing it, then the tetrahedron appears as the difference between two tetrahedra, each having the geodesic ray as an edge and a face of the initial tetrahedron as a base (see Figure~\ref{voltet});
	\item In step \ref{jvertex}, the signs that appear in the sum are the signs of certain determinants;
	\item In step \ref{stepangles}, the angle~$\alpha$ is an angle in a Euclidean triangle and can be computed by elementary trigonometry, and since the upper half-space model is conformal, the angle~$\gamma$ is the Euclidean angle of intersection of the sphere and plane representing the faces of the tetrahedron. 
\end{itemize}
\end{rmks}

The values of the Lobachevsky function are computed with the following lemma. It may be well-known, but we include it for the sake of completeness.
\begin{lem}
 For all~$\theta\in(-\pi,\pi)$ we have the formula
\[\Loba(\theta)=
  \pi\ln\left(\frac{\pi-\theta}{\pi+\theta}\right) + 
  \theta\left(3-\ln\left[2|\theta|\left(1-\Bigl(\frac{\theta}{\pi}\Bigr)^2\right)\right] + 
  \sum_{n=1}^\infty\frac{\zeta(2n)-1}{n(2n+1)}\Bigl(\frac{\theta}{\pi}\Bigr)^{2n}\right)\]
and the bounds
\begin{eqnarray*}
  \sum_{n>r}\frac{\zeta(2n)}{n(2n+1)}\Bigl(\frac{\theta}{\pi}\Bigr)^{2n}	&\le \dfrac{2}{3}\dfrac{1}{1-\Bigl(\dfrac{\theta}{\pi}\Bigr)^{2}}\Bigl(\dfrac{\theta}{\pi}\Bigr)^{2r+2}\\
  \sum_{n>r}\frac{\zeta(2n)-1}{n(2n+1)}\Bigl(\frac{\theta}{\pi}\Bigr)^{2n}	&\le \dfrac{1}{1-\Bigl(\dfrac{\theta}{2\pi}\Bigr)^{2}}\Bigl(\dfrac{\theta}{2\pi}\Bigr)^{2r+2} \cdot
\end{eqnarray*}
\end{lem}

\begin{proof}
 To derive the first expression we use the previous power series expansion and extract the first term of the series expansion of the zeta function. For all~$\theta\in(-\pi,\pi)$ we have
 
\[
  \sum_{n=1}^\infty\frac{\zeta(2n)}{n(2n+1)}\Bigl(\frac{\theta}{\pi}\Bigr)^{2n} = \sum_{n=1}^\infty\frac{1}{n(2n+1)}\Bigl(\frac{\theta}{\pi}\Bigr)^{2n} + \sum_{n=1}^\infty\frac{\zeta(2n)-1}{n(2n+1)}\Bigl(\frac{\theta}{\pi}\Bigr)^{2n}
\]
since all these series converge. We only need to compute the power series that appears. By taking derivatives twice we find that for all~$x\in(-1,1)$ we have
\[\sum_{n=1}^{\infty}\frac{x^{2n+1}}{n(2n+1)} = 2x - x\ln(1-x^2) + \ln\Bigl(\frac{1-x}{1+x}\Bigr)\text{.}\]
Letting~$x=\frac{\theta}{\pi}$ in this expression gives the first formula.
\par To prove the inequalities we are going to bound the values~$\zeta(2n)$ and~$\zeta(2n)-1$ for~$n\ge 1$. By series-integral comparison we get
\[\sum_{k=r}^{\infty}k^{-2n} \le \frac{(r-1)^{1-2n}}{2n-1}\text{,}\]
giving
\[\zeta(2n) = 1 + \sum_{k=2}^{\infty}k^{-2n} \le 1 + \frac{1}{2n-1} \le 2\]
for the first value, and
\[\zeta(2n)-1 = 2^{-2n} + \sum_{k=3}^{\infty}k^{-2n} \le \Bigl(1+\frac{2}{2n-1}\Bigr)2^{-2n} \le 3\cdot 2^{-2n}\]
for the second one. Using these inequalities and the bound~$\frac{1}{n(2n+1)}\le \frac{1}{3}$, and computing the geometric sum gives the result.
\end{proof}

\begin{rmks}~
  \begin{itemize}
  \item With the same method, for any~$k$ we can obtain a formula with remainder term~$O\bigl((\frac{\theta}{k\pi})^{2r}\bigr)$.
  \item In practise, we precompute the coefficients of the power series we are using. By periodicity and oddness, we can always reduce to the case where~$\theta\in [0,\frac{\pi}{2}]$: if the precision is fixed, we know a priori the maximal number of terms needed to evaluate the Lobachevsky function.
 \end{itemize}
\end{rmks}

\subsection{The reduction algorithm}
When we have a fundamental domain, it is natural to ask for an algorithm that, given a point in the hyperbolic 3-space, computes an equivalent point in the fundamental domain and an element in the group that sends one point to the other.

\begin{defi}Let $S$ be a subset of a Kleinian group~$\Gamma$. A point $z\in\hypb$ is~$S$-\emph{reduced} if for all $g\in S$, we have $\dist(z,0)\le\dist(g z,0)$, i.e. if $z\in \overline{\Ext(S)}$.
\end{defi}

\begin{algorithm}[H]
\caption{Reduction algorithm}
\label{redalg}
\begin{algorithmic}[1]
	\REQUIRE A point~$w\in\hypb$, a finite ordered subset~$S\subset\PSL_2(\C)$
	\ENSURE A point~$w'$ and an element~$\delta\in\gp{S}$ s.t. $w'$ is $S$-reduced and $w'=\delta w$
	\STATE $w'\leftarrow w$, $\delta \leftarrow 1$
	\STATE $g\leftarrow 1$
	\REPEAT
		\STATE\label{stepinvar} $w'\leftarrow g w'$, $\delta \leftarrow g\delta$
		\STATE\label{stepchoose} $g\leftarrow$ the first $g\in S$ such that $\dist(g w',0)$ is minimal
	\UNTIL{$\dist(g w',0) \ge \dist(w',0)$}\label{stepcmp}
	\RETURN $w',\delta$
\end{algorithmic}
\end{algorithm}

\begin{prop}
 Given $S$ a finite subset of a Kleinian group~$\Gamma$ and a point $w\in\hypb$, Algorithm \ref{redalg} returns a point $w'$ and $\delta\in\gp{S}$ such that $w'$ is $S$-reduced and $w'=\delta w$.
\end{prop}

\begin{proof}
 After step \ref{stepinvar}, we have $w'=\delta w$ and $\delta\in\gp{S}$. Because of the loop condition, while the algorithm runs the distance~$\dist(w',0)$ decreases. Since $w'$ stays in the orbit of $w'$ under~$\Gamma$ and this orbit is discrete, the algorithm terminates. When it happens, $g$ is an element in~$S$ such that~$\dist(g w',0)$ is minimal and~$\dist(g w',0) \ge \dist(w',0)$, so $w'$ is $S$-reduced.
\end{proof}

\begin{rmk}
	At step \ref{stepchoose}, the $g$ achieving the minimal $\dist(g w',0)$ may not be unique. We can pick any of these elements. Ordering~$S$ gives us a canonical choice.
\end{rmk}

Reducing points can give interesting information about the elements of the group, because if $w$ has a trivial stabilizer, then the orbit map~$\gamma\mapsto \gamma\act w$ is a bijection. This is the reason for introducing the following definition:
\begin{defi}Let $S$ be a subset of a Kleinian group~$\Gamma$ and $w\in\hypb$. An element~$\gamma\in\PSL_2(\C)$ is $(S,w)$-\emph{reduced} if $\gamma w$ is $S$-reduced, i.e. if $\gamma w\in\overline{\Ext(S)}$.
\end{defi}

Given a finite~$S$, $w$ and $\gamma$, we can now compute an $(S,w)$-reduced element $\bar\gamma$ such that $\bar\gamma\equiv \gamma \pmod{S}$ as follows: we reduce $\gamma w$ with respect to $S$; if $\delta\in\gp{S}$ is such that~$\delta (\gamma w)$ is $S$-reduced, then $\bar\gamma = \delta\gamma$ is $(S,w)$-reduced. We also write the reduced element~$\bar\gamma = \reduc_S(\gamma;w)$ and simply~$\reduc_S(\gamma)=\reduc_S(\gamma;0)$. A~priori this reduced element could depend on the chosen ordering in Algorithm~\ref{redalg}.

\begin{prop}\label{propuniquered} Suppose that $\Ext(S)$ is a fundamental domain for $\gp{S}$. Then for $w\in\hypb$ outside of a zero measure, closed subset of~$\hypb$, the following holds: for every~$\gamma\in\Gamma$, there exists a unique $(S,w)$-reduced $\bar\gamma\equiv\gamma\pmod{S}$. If $w\in\Ext(S)$ then $\bar\gamma=1$ if and only if~$\gamma\in\gp{S}$.
\end{prop}

\begin{proof}Let $w\in \Gamma\act\Ext(S)$. The existence follows from Algorithm~\ref{redalg}. For uniqueness, suppose $\bar\gamma$ and $\bar\gamma'$ are $(S,w)$-reduced and $\bar\gamma\equiv\bar\gamma'\equiv\gamma\pmod{S}$. Then~$\bar\gamma w,\bar\gamma' w\in\overline{\Ext(S)}$, and since $w$ is in the orbit of $\Ext(S)$, they are in fact in~$\Ext(S)$. Since these two points are in the same $\gp{S}$-orbit, we have $\bar\gamma=\bar\gamma'$. Now assume~$w\in\Ext(S)$. If~$\bar\gamma=1$ then $\gamma\equiv\bar\gamma\equiv 1\pmod{S}$, i.e. $\gamma\in\gp{S}$. If~$\gamma\in\gp{S}$ then $\gamma\equiv 1\pmod{S}$ and~$1$ is $(S,w)$-reduced so by uniqueness $\bar\gamma=1$.
Moreover the complement of~$\Gamma\act\Ext(S)$ in~$\hypb$ is a locally finite union of faces of~$\Ext(S)$, so it is closed with zero measure.
\end{proof}

Since this provides an algorithm to write an element of the group as a word in the generators and to compute modulo $\gp{S}$ with explicit unique representatives, that particular kind of generating set deserves a name.

\begin{defi} A subset $S$ of a Kleinian group $\Gamma$ is a \emph{basis} if $\Ext(S)$ is a fundamental domain for $\gp{S}=\Gamma$. If $S$ is also a minimal defining set for $\Ext(S)$, it is called a \emph{normalized basis} for $\Gamma$.
\end{defi}

\subsection{Normalized basis algorithms}
Now we describe a general algorithm that computes a normalized basis for a cocompact Kleinian group~$\Gamma$. We will then apply it to arithmetic groups. First note that, after conjugating the group by a suitable element in~$\PSL_2(\C)$, we may assume that~$0\in\hypb$ has a trivial stabilizer in~$\Gamma$ and that every elliptic cycle has length~$1$.

\medskip

We will use two blackbox subalgorithms, Enumerate and IsFullGroup:
\begin{itemize}
  \item Enumerate($\Gamma, n$) takes as an input a positive integer~$n$ and returns a finite set of elements in~$\Gamma$ (the integer~$n$ is a parameter for iteration, it does not have any mathematical meaning);
  \item IsFullGroup($\Gamma,S$) takes as an input a finite normalized basis~$S$ for a subgroup~$\gp{S}\subset\Gamma$ and returns \textbf{true} or \textbf{false} according to whether~$\gp{S}=\Gamma$ or not.
\end{itemize}

In every algorithm, an exterior domain~$\Ext(S)$ with finite~$S$ is represented as a polyhedron in~$\hypb$. We begin with a naive algorithm.

\begin{algorithm}[H]
\caption{Naive normalized basis algorithm}
\label{naivebasisalg}
\begin{algorithmic}[1]
	\REQUIRE A Kleinian group~$\Gamma$
	\ENSURE A normalized basis~$S$ for~$\Gamma$
	\STATE $S \leftarrow \emptyset,\ n\leftarrow 0$
	\REPEAT
		\REPEAT
			\STATE $n\leftarrow n+1$
			\STATE add Enumerate$(\Gamma,n)$ to $S$
			\STATE $S\leftarrow$ minimal defining set of~$\Ext(S)$\label{naivestepnb}
		\UNTIL{$\Ext(S)$ has a face-pairing \AND $\Ext(S)$ is complete \AND $\Ext(S)$ satisfies the cycle condition}
	\UNTIL{IsFullGroup($\Gamma,S$)}
	\RETURN $S$
\end{algorithmic}
\end{algorithm}

We say that~Enumerate is a \emph{complete enumeration} of~$\Gamma$ if we have
\[\bigcup_{n>0}\text{Enumerate}(\Gamma,n)=\Gamma\text{.}\]

\begin{prop}\label{propnaive}If~$\Gamma$ is geometrically finite and Enumerate is a complete enumeration of~$\Gamma$, then Algorithm~\ref{naivebasisalg} terminates after a finite number of steps and the output~$S$ is a normalized basis for~$\Gamma$.
\end{prop}

\begin{proof}The Dirichlet domain centered at~$0$ for~$\Gamma$ has finitely many faces by geometric finiteness. Since Enumerate is a complete enumeration, a defining set for this Dirichlet domain will be enumerated after a finite number of steps. The algorithm will then terminate as all the conditions are satisfied by Dirichlet domains. The output will then be a normalized basis for~$\Gamma$ by Step~\ref{naivestepnb} and Theorem~\ref{poincthm}.
\end{proof}

We will now use the reduction algorithm to improve upon Algorithm~\ref{naivebasisalg}. The main ideas are
\begin{itemize}
	\item reducing the elements that we have to find smaller ones
	\item when the face-pairing condition, the cycle condition or the completeness condition fails, using this fact to find elements that make the exterior domain smaller.
\end{itemize}
For clarity, we divide Algorithm \ref{basisalg} into four routines. Algorithm \ref{basisalg} uses these routines to compute a normalized basis for a geometrically finite Kleinian group~$\Gamma$.

\begin{algorithm}[H]
\caption{Normalized basis algorithm}
\label{basisalg}
\begin{algorithmic}[1]
	\REQUIRE A Kleinian group~$\Gamma$
	\ENSURE A normalized basis~$S$ for~$\Gamma$
	\STATE $S \leftarrow \emptyset,\ n\leftarrow 0$
	\REPEAT
		\REPEAT
			\STATE $n\leftarrow n+1$
			\STATE add Enumerate$(\Gamma,n)$ to $S$
			\STATE $S\leftarrow$ KeepSameGroup($S$)
			\STATE $S\leftarrow$ CheckPairing($S$)
			\STATE $S\leftarrow$ CheckCycleCondition($S$)
			\STATE $S\leftarrow$ CheckComplete($S$)
		\UNTIL{$\Ext(S)$ does not change}
	\UNTIL{IsFullGroup($\Gamma,S$)}
	\RETURN $S$
\end{algorithmic}
\end{algorithm}

\medskip

The first routine, KeepSameGroup, reduces elements as much as possible to eliminate redundant ones and find smaller ones.

\begin{algorithm}[H]
\caption{KeepSameGroup}
\label{ksg}
\begin{algorithmic}[1]
	\REQUIRE A finite subset~$S\subset\PSL_2(\C)$
	\ENSURE A new~$S$ generating the same group with smaller elements
	\REPEAT
		\STATE $U \leftarrow$ minimal defining set of~$\Ext(S)$\label{stepnb}
		\FORALL{$g\in S$}
			\STATE $\bar{g}\leftarrow \reduc_{U}(g)$
			\IF{$\bar{g}\neq \pm 1$}
				\STATE add $\bar{g}$ to $U$
			\ENDIF
		\ENDFOR
		\STATE $S\leftarrow U$
	\UNTIL $\Ext(S)$ does not change
	\RETURN $S$
\end{algorithmic}
\end{algorithm}

\begin{prop}
 If~$S$ is a subset of a Kleinian group, then Algorithm~\ref{ksg} terminates and does not change the group generated by~$S$.
\end{prop}

\begin{proof}
 We first prove the second claim. Every element added to~$S$ belongs to the group generated by~$S$ as it is a reduction by~$U\subset S$ of an element in~$S$. Moreover, every element that is discarded has~$\reduc_{U}(g)=\pm 1$ so at the end of the loop we have~$g\in\gp{S}$, and every other element~$g\in S\setminus U$ is replaced by~$\bar{g}=\reduc_{U}(g)\in \gp{U}g$, so the group generated by~$S$ does not change.

\medskip

Now we prove that the algorithm terminates. First consider the initial~$S$. Let~$M=\max\{\dist(g\act 0,0) : g\in S\}$ and~$X_0=\{g\in\gp{S} : \dist(g\act 0,0)\le M\}$. The set~$X_0$ is finite since~$\gp{S}$ is a Kleinian group, and we have~$S\subset X_0$. By definition of reduction, every element added to~$U$ is in~$X_0$. Moreover, by Step~\ref{stepnb} if an element~$g$ is discarded then its isometric sphere~$\I(g)$ does not intersect~$\overline{\Ext(S)}$, so~$g\act 0$ is in the complement of~$\Ext(S)$: $g$ cannot be the reduction of any element, so it cannot be added again. Similarly if~$g\in S\setminus U$ is replaced by~$\bar{g}\neq g$, then~$g$ is not reduced so it cannot be added again. Hence the algorithm terminates.
\end{proof}

The second routine, CheckPairing, checks whether~$\Ext(S)$ has a face-pairing. If it does not, it finds elements that make~$\Ext(S)$ smaller.

\begin{algorithm}[H]
\caption{CheckPairing}
\label{cp}
\begin{algorithmic}[1]
	\REQUIRE A finite subset~$S\subset\PSL_2(\C)$
	\ENSURE A new~$S$ such that~$\Ext(S)$ is smaller if it did not have a face-pairing
	\STATE $S\leftarrow S\cup S^{-1}$
	\FORALL{$e$ edge in $\I(g)$ and $g\in S$, s.t. $ge$ not an edge of~$\overline{\Ext(S)}$}\label{steptestpaired}
		\STATE $x\leftarrow$ $x\in e$ such that $g x\notin \overline{\Ext(S)}$
		\STATE $\bar{g}\leftarrow \reduc_S(g;x)$
		\STATE \label{steppair}add~$\bar{g}, \bar{g}^{-1}$ to~$S$
	\ENDFOR
	\RETURN $S$
\end{algorithmic}
\end{algorithm}

\begin{prop}
 If~$\Ext(S)$ does not have a face-pairing, then after applying Algorithm~\ref{cp}, $\Ext(S)$ is strictly smaller.
\end{prop}

\begin{proof}
 If there is a nonpaired edge, at Step \ref{steppair}, since $x\in\I(g)$ we have~$\dist(g x,0)=\dist(x,0)$ and since $g x\notin\overline{\Ext(S)}$ we have~$\dist(g x,0)>\dist(\bar g x,0)$. Putting these two together gives $\dist(\bar g x,0)<\dist(x,0)$, i.e. $x\in\Int(\bar g)$ so finally we have~$\Ext(S\cup\{\bar g\})\varsubsetneq \Ext(S)$.
\end{proof}

We give a second possible algorithm for CheckPairing. It is simpler but less efficient in practice. It uses the fact that if a non-elliptic cycle has length three (which is generically the case), then it is of the form~$e\subset\I(g)\cap\I(h)$, $g e\subset\I(g^{-1})\cap\I(gh^{-1})$, $h e\subset\I(hg^{-1})\cap\I(h^{-1})$.

\begin{algorithm}[H]
\caption{CheckPairing'}
\label{cp2}
\begin{algorithmic}[1]
	\REQUIRE A finite subset~$S\subset\PSL_2(\C)$
	\ENSURE A new~$S$ such that~$\Ext(S)$ is smaller if it did not have a face-pairing
	\STATE $S\leftarrow S\cup S^{-1}$
	\FORALL{$g,h\in S$ s.t. $\I(g)\cap\I(h)\neq\emptyset$ and~$h\neq g^{-1}$}
		\STATE add~$gh^{-1}, hg^{-1}$ to~$S$
	\ENDFOR
	\RETURN $S$
\end{algorithmic}
\end{algorithm}

\begin{prop}\label{propcp2} 
If~$\Ext(S)$ does not have a face-pairing, then after applying Algorithm~\ref{cp2}, $\Ext(S)$ is strictly smaller.
\end{prop}

\begin{proof}
 If there is a nonpaired edge, then there exists elements~$g,h\in S$ in the minimal defining set of~$\Ext(S)$ and a point~$x\in \I(g^{-1})\cap\overline{\Ext(S)}$ such that~$g^{-1}x\in\Int(h)$ (so that~$h\neq g^{-1}$). Since we also have~$g^{-1}x\in\I(g)$ and~$\I(g)$ is not contained in~$\Int(h)$, we get~$\I(g)\cap\I(h)\neq\emptyset$, so these elements will be considered in the loop. On the other hand we have~$\dist(x,0)=\dist(g^{-1}x,0) > \dist(g^{-1}hx,0)$, so~$x\in\Int(g^{-1}h)$: we have~$\Ext(S\cup\{g^{-1}h\})\varsubsetneq \Ext(S)$.
\end{proof}

\begin{rmk}
 Although this algorithm is less efficient than Algorithm~\ref{cp}, it is interesting as it gives a geometric understanding of the method described in~\cite{lip}: \emph{``we consider words that are two-word combinations of those forming the sides of the existing domain to modify the domain. (...) This procedure has proven to be fast and effective in practice.''} Proposition~\ref{propcp2} explains why taking products of two elements forming the sides of the domain is useful, and in Algorithm~\ref{cp2} we get a geometric description of the the products that we should form. Actually, the computation in the proof of Proposition~\ref{cp2} also shows that if~$\I(g^{-1}h)$ reduces~$\Ext(S)$, then~$\I(g)\cap\I(h)\neq\emptyset$.
\end{rmk}

The third routine, CheckCycleCondition, checks whether~$\Ext(S)$ satisfies the cycle condition. If it does not, it finds elements that make~$\Ext(S)$ smaller.

\begin{algorithm}[H]
\caption{CheckCycleCondition}
\label{ccc}
\begin{algorithmic}[1]
	\REQUIRE A finite subset~$S\subset\PSL_2(\C)$
	\ENSURE A new~$S$ s.t.~$\Ext(S)$ is smaller if it did not satisfy the cycle condition
	\STATE Compute every well-defined edge cycle
		\FORALL{$g$ cycle transformation for the edge~$e$}
			\IF{$g\neq \pm 1$ fixes at most one point in~$e$}\label{stepbegintorsion}
				\STATE\label{addloxo}$S\leftarrow S\cup\{g,g^{-1}\}$
			\ELSIF{$g\neq \pm 1$ fixes every point in~$e$}
				\STATE $S\leftarrow S\cup\gp{g}$\label{stepaddellstab}
			\ELSE
				\STATE $m\leftarrow$ length of the cycle\label{stepbegincycle}
				\FORALL{$0<i<m$}
					\STATE $h\leftarrow g_i\dots g_1$\label{stepoverlap}
					\STATE \label{stepshrink}add~$h, h^{-1}$ to~$S$
				\ENDFOR\label{stependcycle}
			\ENDIF
		\ENDFOR
	\RETURN $S$
\end{algorithmic}
\end{algorithm}

\begin{rmks}~
\begin{itemize}
 \item If we assume that every non-elliptic cycle has length three, then the steps~\ref{stepbegincycle}--\ref{stependcycle} are unnecessary, as in this case the partial cycle transformations at an edge contained in~$\I(g)\cap\I(h)$ are~$g$, $h=(hg^{-1})g$, $1=h^{-1}(hg^{-1})g$.
 \item If we know in advance that the group~$\Gamma$ is torsion-free, then we can omit the steps~\ref{stepbegintorsion}--\ref{stepaddellstab}.
 \item Assuming both, we can omit CheckCycleCondition completely.
\end{itemize}
\end{rmks}

\begin{lem}\label{lemshrinkdom}Suppose~$S\subset\Gamma$ is a subset of a Kleinian group~$\Gamma$ such that~$0$ has a trivial stabilizer in~$\Gamma$, and suppose there is an element~$h\in\Gamma\setminus\{\pm 1\}$ and a point~$x\in\Ext(S)$ such that~$h x\in\Ext(S)$. Then~$\Ext(S\cup\{h,h^{-1}\})\varsubsetneq \Ext(S)$.
\end{lem}

\begin{proof}
 First suppose that~$\dist(x,0)<\dist(hx,0)$. Then writing~$x=h^{-1}(hx)=h^{-1}y$ we get~$\dist(h^{-1}y,0)<\dist(y,0)$ i.e.~$y\in\Int(h)$. Since we also have~$y\in\Ext(S)$, we obtain~$\Ext(S\cup\{h\})\varsubsetneq \Ext(S)$.

\medskip

Othewise we have~$\dist(hx,0)\le\dist(x,0)$. This means that~$x\in\overline{\Int(h^{-1})}$, but since~$x\in\Ext(S)$ we get~$\Ext(S\cup\{h^{-1}\})\varsubsetneq \Ext(S)$.
\end{proof}

\begin{prop}
 If~$\Ext(S)$ does not satisfy the cycle condition, then after applying Algorithm~\ref{ccc}, $\Ext(S)$ is strictly smaller.
\end{prop}

\begin{proof}
 Since the cycle transformation at an edge stabilizes it, if the edge is not equal to a geodesic then the cycle transformation fixes it pointwise and condition~(i) is automatically satisfied. Suppose that there is a cycle for an edge~$e$ equal to a geodesic and that does not satisfy condition~(i), and let~$g$ be the corresponding cycle transformation. Then the transformation~$g$ is either loxodromic, or elliptic of order~$2$ with exactly one fixed point in~$e$. In both cases, Step~\ref{addloxo} is executed. In the first case, since the interior of the isometric sphere of a loxodromic element contains one of its fixed points and the interior of the isometric sphere of its inverse contains the other, we have~$\overline{\Ext(\{g,g^{-1}\})}\cap e\varsubsetneq e$ so~$\Ext(S\cup\{g,g^{-1}\})\varsubsetneq \Ext(S)$. In the second case, the edge~$e$ contains exactly one fixed point of~$g$ in~$\hyp^3$, so we again have~$\overline{\Ext(\{g\})}\cap e\varsubsetneq e$ and we get~$\Ext(S\cup\{g,g^{-1}\})\varsubsetneq \Ext(
S)$.

\medskip

Now suppose some cycle angle for a non-elliptic cycle is larger than~$2\pi$. Then considering the images~$P=\Ext(S),g_1^{-1} P,\dots,(g_i\dots g_1)^{-1} P$ of~$P=\Ext(S)$ that glue one after another around~$e$, we see that there is an overlap: there exists a point~$x\in P$ such that~$hx\in P$ for some~$h$ considered in Step~\ref{stepoverlap}. In this case after Step~\ref{stepshrink} we have~$\Ext(S\cup\{h,h^{-1}\})\varsubsetneq \Ext(S)$ by Lemma~\ref{lemshrinkdom}. Since the cycle transformation is the identity, the angle cannot be smaller than~$2\pi$.

\medskip

Finally suppose some cycle angle for an elliptic cycle at an edge~$e$ with cycle transformation~$g$ with order~$\nu$ does not satisfy condition~(ii). The cycle has length~$1$, so~$e\subset\I(g)\cap\I(g^{-1})$, and the angle at~$e$ is a multiple of~$\frac{2\pi}{\nu}$. After running Step~\ref{stepaddellstab} the domain $\Ext(\{g,g^{-1}\})$ is replaced by the Dirichlet domain of the finite group~$\gp{g}$, which satisfies the cycle condition, so the new angle at~$e$ is equal to~$\frac{2\pi}{\nu}$.
\end{proof}

The fourth routine, CheckComplete, checks whether~$\Ext(S)$ is complete. If it is not, it finds elements that make~$\Ext(S)$ smaller.

\begin{algorithm}[H]
\caption{CheckComplete}
\label{chcomplete}
\begin{algorithmic}[1]
	\REQUIRE A finite subset~$S\subset\PSL_2(\C)$
	\ENSURE A new~$S$ such that~$\Ext(S)$ is smaller if it was not complete
	\STATE Compute every tangency vertex cycle
		\FORALL{$h$ tangency vertex transformation}
			\IF{$h\neq 1$ is loxodromic}
				\STATE\label{addloxotv}add~$h,h^{-1}$ to~$S$
			\ENDIF
		\ENDFOR
	\RETURN $S$
\end{algorithmic}
\end{algorithm}

\begin{rmk}
 If we know in advance that the group~$\Gamma$ is cocompact, we can omit CheckComplete in Algorithm~\ref{basisalg} and simply test whether~$\Ext(S)$ is bounded.
\end{rmk}

\begin{prop}
 If~$\Ext(S)$ is not complete, then after applying Algorithm~\ref{ccc}, $\Ext(S)$ is strictly smaller.
\end{prop}

\begin{proof}
 If~$h$ is a tangency vertex transformation at~$z=\I(g)\cap\I(g')\in\partial{\hypb}$, then it fixes~$z$. By looking at the successive images of the polyhedron along the cycle we see that~$\I(g')$ separates~$\I(g)$ from~$h\I(g)$, so~$h$ has infinite order. If~$\Ext(S)$ is not complete, then~$h$ is loxodromic. Being a fixed point of~$h$, the point~$z$ is contained in~$\Int(h)\cup\Int(h^{-1})$, so we get~$\Ext(S\cup\{h,h^{-1}\})\varsubsetneq \Ext(S)$.
\end{proof}

\begin{prop}
 \label{propbasisalg}Let $\Gamma$ be a Kleinian group. The following holds for Algorithm~\ref{basisalg} applied to~$\Gamma$:
\begin{enumerate}[(i)]
	\item\label{correctsg} Suppose the algorithm terminates. Then the output is a normalized basis for~$\Gamma$.
	\item\label{termtriv} Suppose that~$\Gamma$ is geometrically finite and Enumerate is a complete enumeration of~$\Gamma$. Then the algorithm terminates.
\end{enumerate}
\end{prop}

\begin{rmk}\label{rmkterm}
In practise Algorithm~\ref{basisalg} runs much faster that the naive Algorithm~\ref{naivebasisalg} (see section~\ref{cmpnb}), but unfortunately we could not prove it. What we believe is that in Algorithm~\ref{basisalg} the blackbox Enumerate only needs to find a set of generators for the group, and then the other routines find the elements of the normalized basis; in Algorithm~\ref{naivebasisalg} the blackbox Enumerate needs to find directly the elements of the normalized basis, which is harder. The natural idea would be to put the routines in a loop that would not contain Enumerate in Algorithm~\ref{basisalg}, but then it is not clear whether this internal loop would terminate; actually in general it is false, since~$\Gamma$ may admit finitely generated subgroups that are not geometrically finite.
\end{rmk}

\smallskip

\begin{proof}~
 \begin{enumerate}[(i)]
  \item If the algorithm terminates, then by Theorem \ref{poincthm}, since~$\Ext(S)$ is complete, has a face-pairing and satisfies the cycle condition, the set~$S$ is a normalized basis for~$\gp{S}$. It is then valid to use~IsFullGroup to check that~$\gp{S}=\Gamma$.
  \item The closure of the Dirichlet domain centered at~$0$ for~$\Gamma$ has finitely many faces by geometric finiteness. Since Enumerate is a complete enumeration, a defining set for this Dirichlet domain will be enumerated after a finite number of steps. The algorithm will then terminate as all the conditions are satisfied by the Dirichlet domain.
 \end{enumerate}

\end{proof}


\subsection{Instantiation of the blackboxes}



\subsubsection{Enumerate and IsFullGroup for a group given by generators}

Suppose the group~$\Gamma$ is given by a finite set of generators~$G$. We can take for Enumerate the algorithm that writes every word of length~$n$ in the generators, and we can take for IsFullGroup the algorithm that reduces every element in~$G$ with respect to the given normalized basis~$S$ and returns whether every generator reduces to~$\pm 1$: by Proposition \ref{propuniquered}, this is equivalent to~$\Gamma\subset\gp{S}$.

\subsubsection{Enumerate and IsFullGroup for an arithmetic group}

We provide a possible instantiation of the blackboxes Enumerate and IsFullGroup for an arithmetic group~$\Gamma(\order)$ attached to a maximal order~$\order$ in a Kleinian quaternion algebra~$B$ with base field~$F$ of degree~$n$.

\par We describe IsFullGroup first. A subgroup is proper if and only if its covolume is infinite or at least twice the covolume of~$\Gamma$, the quotient of the covolumes being the index of the subgroup. Since~$\Gamma$ comes from a maximal order, the covolume of~$\Gamma$ is given by~\eqref{covolumeformula}, which we can compute, and the covolume of a subgroup can be computed with Algorithm~\ref{algovol} once we have a normalized basis. We take for IsFullGroup the algorithm that computes the covolume~$\covol(\Gamma)$ by the formula and the volume~$V$ of~$\Ext(S)$ for the given normalized basis~$S$, and returns whether~$\frac{V}{\covol(\Gamma)} < 2$. Since~$S$ is a normalized basis for~$\gp{S}$, the polyhedron~$\Ext(S)$ is a fundamental domain for~$\gp{S}$ so the volume~$V$ equals the covolume of~$\gp{S}$.

\medskip

We now describe an instantiation of the blackbox Enumerate for the Kleinian group associated with an order~$\order$ in~$B$. 
Under the natural embedding~$\order \subset B\hookrightarrow B\otimes_{\Q}\R$, the order~$\order$ is discrete. Now suppose that we have a positive definite quadratic form~$Q : B\otimes_{\Q}\R \to \R$. Then~$\order$ becomes a full lattice in a real vector space of dimension~$4n$. We can use lattice enumeration algorithms such as the Kannan-Fincke-Pohst algorithm~\cite{fp, kannan} to enumerate elements in~$\order$ that are short with respect to~$Q$. We can then select the elements having reduced norm~$1$. As we increase the bound on the values of~$Q$, we will get every element in~$\order^1$. A priori any such quadratic form would work, but here we describe one that has a geometric meaning.

\medskip

Recall we can embed~$B$ in~$M_2(\C)$ in such a way that~$\order^1$ becomes discrete in~$\SL_2(\C)$. This embedding is only defined up to conjugation by an element of~$\PSL_2(\C)$. Let~$\rho$ be such an embedding. If~$B=\quatalg{a,b}{F}$ we can take for example
\[\rho: x+yi+zj+tij\mapsto \begin{pmatrix}x+y\alpha&z+t\alpha\\(z-t\alpha)\beta&x-y\alpha\end{pmatrix}\]
where~$\sigma$ is a complex embedding of~$F$, $\beta=\sigma(b)$ and~$\alpha$ is a square root of~$\sigma(a)$.

\par For~$m=\sm{a&b\\c&d}\in\mat_2(\C)$, we define~$\invrad(m)=\left|(c+\bar{b})+(d-\bar{a})j\right|^2$.

\begin{prop}The quadratic form $Q : B\otimes\R \to \R$ defined by
\[Q(x)=\invrad(\rho(x))+\tr_{F/\Q}(\nrd(x))\text{ for all } x\in B\]
is positive definite and satisfies
\[Q(x)=\frac{4}{\rad(\rho(x))^2}+n \text{ for all } x\in\order^1\]
where~$\rad(g)$ denotes the Euclidean radius of the isometric sphere of~$g\in\SL_2(\C)$ if~$g\act 0\neq 0$, and~$\infty$ otherwise.
\end{prop}

\begin{proof}
We show first that~$Q$ is positive definite. For a matrix~$m\in\mat_2(\C)$ we have~$\invrad(m) = |c+\bar{b}|^2+|d-\bar{a}|^2 = \|m\|^2 - 2\,\Re(\det m)$ where $\|\cdot\|$ is the usual $L^2$ norm on $\mat_2(\C)$, so that $\|\cdot\|^2$ is a positive definite quadratic form on $\mat_2(\C)$. Since~$\nrd$ is a positive definite quadratic form on~$\Hamil$ and we have the decomposition~$B\otimes\R\cong\mat_2(\C)\oplus\Hamil^{n-2}$, we can construct a positive definite quadratic form on $B\otimes\R$ by letting for all $x\in B\otimes\R$
	\[Q(x)=\|m\|^2+\nrd(h_1)+\dots+\nrd(h_{n-2})=\invrad(m)+\tr_{F\otimes\R/\R}(\nrd(x))\]
	where
	\[x=m+h_1+\dots+h_{n-2}\in \mat_2(\C)\oplus\Hamil^{n-2}\text{,}\]
	since $2\,\Re(\det m)+\nrd(h_1)+\dots+\nrd(h_{n-2})=\tr_{F\otimes\R/\R}(\nrd(x))$.
	This gives the positive definiteness.

\medskip
	
	For the formula on~$\order^1$, note that according to~\eqref{eqradius}, it is
	\[\invrad(g)=\left|(c+\bar{b})+(d-\bar{a})j\right|^2=\frac{4}{\rad(g)^2}\]
	for $g\in\SL_2(\C)$ not fixing $0$ in $\hypb$, and if $g$ fixes $0$ then $\invrad(g)=0$.
\end{proof}

We obtain the following enumeration algorithm. It is a complete enumeration of~$\Gamma(\order)$, and depends on a parameter: a sequence of bounds~$A_n\to\infty$.

\begin{algorithm}[H]
\caption{Enumerate}
\label{enumbigball}
\begin{algorithmic}[1]
	\REQUIRE A positive integer~$n$
	\ENSURE A finite subset~$L\subset\Gamma(\order)$
	\STATE $L\leftarrow \emptyset$
	\FORALL{$x\in\order$ such that~$Q(x)\le A_n$}
		\IF{$\nrd(x)=1$}
			\STATE Add $\rho(x)$ to~$L$
		\ENDIF
	\ENDFOR
	\RETURN $L$
\end{algorithmic}
\end{algorithm}

We are now going to present a probabilistic enumeration algorithm. It is not a complete enumeration, but performs better in pratice (see section~\ref{exenum}). It uses variants of the former quadratic form.

\begin{defi}
 Let~$z_1,z_2\in\hyp^3$. Let~$h_1,h_2\in\SL_2(\C)$ be such that~$z_1=h_1\act j$ and~$z_2=h_2\act j$. We then define the quadratic form~$Q_{z_1,z_2}$ by
\[Q_{z_1,z_2}(x)=\invrad(h_2^{-1}\rho(x)h_1)+\tr_{F/\Q}(\nrd(x))\]
for all~$x\in B$.
\end{defi}

This family of quadratic forms has the following properties.

\begin{prop}
 Let~$z_1,z_2\in\hyp^3$. Then~$Q_{z_1,z_2}$ does not depend on the choice of~$h_1,h_2\in\SL_2(\C)$ such that~$z_1=h_1\act j$ and~$z_2=h_2\act j$. It is positive definite, and for all~$g\in\order^1$ we have
\[Q_{z_1,z_2}(g) = 2\cosh\dist(gz_1,z_2)-2+n\text{.}\]
\end{prop}

\begin{proof}
 The matrices~$h_1$ and~$h_2$ are defined up to right multiplication by~$\SU_2(\C)$, the stabilizer of the point~$j$. For all matrices~$m\in\mat_2(\C)$ we have~$\invrad(m) = \|m\|^2 - 2\,\Re(\det m)$, which is not changed by left and right multiplication of~$m$ by elements of~$\SU_2(\C)$, so that~$Q_{z_1,z_2}$ does not depend on the choice of~$h_1$ and~$h_2$.

\par For~$z_1=z_2=j$ the formula reads~$\|g\|^2 = 2\cosh\dist(g j, j)$ for all~$g\in\SL_2(\C)$, which is well-known (and is a direct consequence of the explicit formulas for the hyperbolic distance). Then for arbitrary~$z_1,z_2\in\hyp^3$ we have
\[\|h_2^{-1}gh_1\|^2 = 2\cosh\dist(h_2^{-1}gh_1 j, j) = 2\cosh\dist(gh_1 j, h_2j)=2\cosh\dist(g z_1, z_2)\text{.}\]
\end{proof}

This family of quadratic forms is very useful, as it enables us to determine the elements~$g\in\Gamma(\order)$ such that~$gz_1$ is close to~$z_2$. We propose the following probabilistic algorithm for enumerating elements in~$\Gamma(\order)$. It depends on a choice of some parameters: an increasing sequence of positive numbers~$R_n\to\infty$ representing the radius of the search space, a sequence of positive integers~$N_n\in\Z_{>0}$ representing the number of enumerations in small balls, and a positive number~$A$ being a bound on the quadratic form. For~$w_1,w_2\in\hypb$, we write~$Q_{w_1,w_2}=Q_{\eta^{-1}(w_1),\eta^{-1}(w_2)}$.

\begin{algorithm}[H]
\caption{Enumerate'}
\label{enumsmallballs}
\begin{algorithmic}[1]
	\REQUIRE An positive integer~$n$
	\ENSURE A finite subset~$L\subset\Gamma(\order)$
	\STATE $L\leftarrow \emptyset$
	\FOR{$i=1$ to $N_n$}
		\STATE Draw a point~$w\in\hypb$ such that~$\dist(0,w)\le R_n$ randomly, uniformly w.r.t. the hyperbolic volume
		\FORALL{$x\in\order$ such that~$Q_{0,w}(x)\le A$}
			\IF{$\nrd(x)=1$}
				\STATE Add $\rho(x)$ to~$L$
			\ENDIF
		\ENDFOR
	\ENDFOR
	\RETURN $L$
\end{algorithmic}
\end{algorithm}

\begin{rmks}~
\begin{itemize}
  \item We can also use these quadratic forms differently: if we miss an element of the group to ``close off'' the exterior domain around a point at infinity~$\xi$, we can look for elements of small~$Q_{j,z}$ where~$z\to\xi$. This is a similar idea as in Remark 4.9 in~\cite{voightfuchsian}, but the quadratic form that was used there is the analogue of~$Q_{z,z}$. If~$g$ is the element that we are looking for, $\dist(gz,z)$ is bounded by below by a positive constant if~$g$ is loxodromic, which is the generic case. On the contrary we have~$\dist(g j,z)\to 0$ as~$z\to g j$.
  \item The efficiency of this algorithm depends on the choice of the parameters~$N_n$, $R_n$ and~$A$. Heuristics led us to the following choice, which works well in practise:
  \begin{itemize}
    \item we use a small bound~$A=\alpha\cdot |\Delta_F N(\Delta_B)|^{\frac{1}{4[F:\Q]}}$ so that the number of~$x\in\order$ such that~$Q_{0,w}(x)\le A$ is approximately constant by Gaussian heuristic;
    \item experimental evidence and~\cite[Theorem 1.5]{generators} suggest that a number of random elements of~$\Gamma$ proportional to~$\covol(\Gamma)$ has a good probability to generate~$\Gamma$, and by Gaussian heuristic we need $O(\covol(\Gamma))$ random centers to obtain one element of the group on average, so we choose~$N_0 = \beta \cdot \covol(\Gamma)^2$, and we increase it exponentially fast:~$N_n = (1+\eta)^n N_0$;
    \item the radius~$R_n$ has to be large enough to ensure good randomness of the elements of~$\Gamma$, so we choose~$R_0$ such that~$\vol(\ball(w,R_0)) = \covol(\Gamma)^\gamma$ and we increase it in arithmetic progression (so the volume increases exponentially fast):~$R_n=R_0+\epsilon\cdot n$. Because of our choice of~$N_n$ we take~$\gamma>2$.
  \end{itemize}
\end{itemize}
\end{rmks}

Now we explain how we draw points at random in the ball~$\ball(0,R)$ of radius~$R$. Since the hyperbolic volume is invariant by rotation around~$0$, it is equivalent to draw a random point uniformly on the sphere, and then multiply it by an appropriate random scalar independent from the point on the sphere. Thus we only have to determine the distribution of the distance from~$0$ of the points in the ball of radius~$R$. Let~$X$ be a random variable with uniform distribution in~$\ball(0,R)$. The cumulative distribution function of the distance to~$0$ is
\[f_R(r) = \proba_X(\dist(X,0)\le r)=\frac{\vol(\ball(0,r))}{\vol(\ball(0,R))}\text{.}\]
Recall that the volume~$v(r)$ of the ball of radius~$r$ is~$v(r)=\pi(\sinh(2r)-2r)$. It is clear that the function~$f_R : [0,R]\to [0,1]$ is a continuous bijection. It implies that~$\dist(0,X)=f_R^{-1}(U)$ where~$U$ is a uniform random variable in~$[0,1]$. We rewrite that expression as~$\dist(0,X)=v^{-1}(U')$ where~$U'$ is a uniform variable in~$[0,v(R)]$. It is well-known how to draw a uniform variable in an interval and on a sphere, and~$v^{-1}$ can be computed by Newton iteration.

\subsection{Floating-point implementation}\label{float-impl}

Here we describe a floating-point implementation of the above algorithms. We start with a lemma giving us control on the error made when having an element of the group act on a point. We only study the stability of the algorithm, so we do not take into account the error made by rounding in elementary operations.


\begin{lem}\label{lemapprox}
 Let~$g\in\SL_2(\C)$, $\tg\in\mat_2(\C)$ and~$w,\tw\in\hypb$. Let~$\epsilon=|w-\tw|$, $\eta=\|g-\tg\|$ and~$\delta=\frac{1}{1-|w|^2}$. Suppose that~$(\nmg\epsilon+2\eta)^2\le\frac{1}{3\delta}$. Then the quantity~$\widetilde{gw}$ obtained by applying Formula~\eqref{actionformula} to~$\tg$ and~$\tw$ is well-defined, and we have
\[|g\act w - \widetilde{gw}|\le 68\,\delta^{\frac32}\nmg^3\epsilon + 136\,\delta^{\frac32}\nmg^2\eta\text{.}\]
\end{lem}
\begin{proof}
By direct computation we have~$|A-\tA|\le\sd\eta$ and~$|A|\le \sd\nmg$, and the same inequalities for~$B,C,D$. We write
\[g\act w = (Aw+B)(Cw+D)^{-1}=\frac{1}{|Cw+D|^2}(Aw+B)(\overline{w}\overline{C}+\overline{D})\]
and similarly for~$\tg,\tw$. Another direct computation gives
\begin{equation}\label{eqrad}
|w|^2-|g\act w|^2 = \left(1-\frac{4}{|Cw+D|^2}\right)(|w|^2-1)\text{,}
\end{equation}
showing that
\[\frac{1}{|Cw+D|^2}\le \frac14(1 + 2\delta)\le\frac34\delta\ \text{ and }\ \frac{4}{3\delta} \le |Cw+D|^2\text{.}\]
By the triangle inequality, adding and substracting~$A\tw$ gives
\[|Aw-\tA\tw|\le \sd\nmg\epsilon+\sd\eta\]
and the same inequality for~$Cw$. We get
\[|(Cw+D)-(\tC\tw+\tD)|^2 \le 2(\nmg\epsilon+2\eta)^2\le \frac{|Cw+D|^2}{2}\]
since by hypothesis we have~$(\nmg\epsilon+2\eta)^2\le\frac{1}{3\delta}$. In particular~$\tC\tw+\tD\neq 0$ and~$\widetilde{gw}$ is well-defined. By the mean value theorem this gives
\[||Cw+D|^{-2}-|\tC\tw+\tD|^{-2}|\le (6\delta)^{\frac32}(\nmg\epsilon+2\eta)\]
We also get
\begin{eqnarray*}
&|(Aw+B)(\overline{w}\overline{C}+\overline{D}) - (\tA\tw+\tB)(\overline{\tw}\overline{\tC}+\overline{\tD})|\\
\le &|Aw+B|(\sd\nmg\epsilon+2\sd\eta)+2|Cw+D|(\sd\nmg\epsilon+2\sd\eta)\\
\le &(2\sd\nmg)(\sd\nmg\epsilon+2\sd\eta)+(2\sd\nmg)(2\sd\nmg\epsilon+4\sd\eta)\\
= &12\nmg^2\epsilon + 24\nmg\eta\text{.}
\end{eqnarray*}
Finally we have
\begin{eqnarray*}
&|(Aw+B)(Cw+D)^{-1}-(\tA\tw+\tB)(\tC\tw+\tD)^{-1}|\\
\le & |g\act w||Cw+D|^2(6\delta)^{\frac32}(\nmg\epsilon+2\eta)+\frac{2}{|Cw+D|^2}(12\nmg^2\epsilon + 24\nmg\eta)\\
\le & (24\rs+9)\delta^{\frac32}\nmg^3\epsilon + (48\rs+18)\delta^{\frac32}\nmg^2\eta\\
\le & 68\,\delta^{\frac32}\nmg^3\epsilon + 136\,\delta^{\frac32}\nmg^2\eta
\end{eqnarray*}
as claimed.
\end{proof}

In the following, we want to maintain the property~$(\nmg\epsilon+2\eta)^2\le\frac{1}{3\delta}$ for every element~$g$ and every point~$w$ considered, where~$\epsilon$ is the imprecision on the points in~$\hypb$, $\eta$ the imprecision on the elements~$g$ considered, and~$\eta=\frac83\epsilon$.

We now describe the modification of the algorithms for the floating-point version. In the reduction algorithm (Algorithm~\ref{redalg}), we choose~$\alpha>0$ and in Step~\ref{stepcmp} we replace the inequality~$\dist(g w',0) \ge \dist(w',0)$ by~$\frac{4}{|Cw'+D|^2}\le 1+\alpha$. Since we have~$w'\in\Ext(g)$ if and only if~$|Cw'+D|^2 \ge 4$, the modified condition is indeed an approximation of the exact condition.

\begin{prop}
 Let~$\beta = \alpha - 68\,\delta^{\frac52}M^3\epsilon - 136\,\delta^{\frac52}M^2\eta$ where~$\delta=\frac1{1-|w|^2}$ and~$M=\max_{g\in S}\|g\|$. If~$\beta>0$, then the floating-point version of the reduction algorithm terminates.
\end{prop}

\begin{proof}
 Formula~\eqref{eqrad} may be rewritten
\[1-|g\act w|^2 = \frac{4}{|Cw+D|^2}(1-|w|^2)\text{,}\]
which gives, if the modified condition of Step~\ref{stepcmp} is not satisfied
\[1-|g\act w'|^2 \ge (1+\alpha)(1-|w'|^2)\text{.}\]
Lemma~\ref{lemapprox} gives
\[1-|\widetilde{gw}'|^2 \ge (1+\beta)(1-|w'|^2)\text{,}\]
so~$1-|w'|^2$ is multiplied by~$1+\beta$ at each step of the algorithm. Since we also have~$1-|w'|^2\le 1$, the algorithm terminates.
\end{proof}

We want to use a uniform~$\alpha$ that tends to~$0$ as~$\epsilon\to 0$. For this, we assume that we only consider points~$w$ such that~$1-|w|^2\ge 2\epsilon^{\frac29}$ and elements~$g$ such that~$\nmg\le \epsilon^{-\frac19}$. Assuming that~$\epsilon<10^{-9}$ we can then take~$\alpha = 18\epsilon^{\frac19}$. These assumptions also ensure that~$(\nmg\epsilon+2\eta)^2\le\frac{1}{3\delta}$, and are compatible since the points~$g\act 0$ that we have to consider satisfy~$1-|g\act 0|^2 = \frac{4}{\nmg^2+2}\ge \frac{2}{\nmg^2}\ge 2\epsilon^{\frac29}$.

There is no change in KeepSameGroup (Algorithm~\ref{ksg}): the same argument shows that the algorithm terminates, regardless of finite precision in the computations. 

The routine CheckPairing (Algorithm~\ref{cp}) should only consider an edge~$e$ contained in~$\I(g)$ as not being paired if there is~$x\in e$ and we have the stronger inequality~$|C\widetilde{gx}+D|^2 < \frac{4}{1+\alpha}$ and~$C,D$ correspond to~$h$ for some~$h\in S$. This ensures that the floating-point reduction will yield a non-trivial element, since at least one step of reduction will be performed. 

The routines CheckCycleCondition (Algorithm~\ref{ccc}) and CheckComplete (Algorithm~\ref{chcomplete}) contain only finite loops regardless of the use of finite precision, so there is no change in them. 


\begin{prop}
 The floating-point version of the Normalized basis algorithm (Algorithm~\ref{basisalg}) terminates.
\end{prop}

\begin{proof}
 By the arguments above, each of the routines terminates. Moreover, because of precision restriction we impose~$\nmg\le\epsilon^{-\frac19}$ for every element~$g$ of the group considered in the algorithm, so that only finitely many~$g$ can be used, so the algorithm terminates.
\end{proof}

Of course if the precision chosen is insufficient, the algorithm may terminate with an error or a wrong answer, but with Riley's methods~\cite{riley}, we can use Poincar\'e's theorem with the approximate fundamental domain to prove that the computed presentation is correct. Alternatively, we could check the fundamental domain algebraically, but this is likely to be time-consuming.



\subsection{Master algorithm}

As a summary, this is our master algorithm for computing an arithmetic Kleinian group associated with a maximal order.
\begin{algorithm}[H]
\caption{Master algorithm}
\label{masteralgo}
\begin{algorithmic}[1]
	\REQUIRE A maximal order~$\order$ in a Kleinian quaternion algebra~$B$
	\ENSURE A finitely presented group~$G$, and two computable group homomorphism~$\phi: G\to\Gamma(\order)$ and~$\psi : \Gamma(\order)\to G$, inverse of each other
	\STATE Choose an embedding~$\rho: B\hookrightarrow \mat_2(\C)$ s.t. the point~$0$ has trivial stabilizer in the group~$\Gamma(\order)=\rho(\order^1)/\{\pm 1\}$
	\STATE $V \leftarrow \covol(\Gamma(\order))$ computed with Formula~\eqref{covolumeformula}
	\WHILE{IsFullGroup($S$)}
	    \STATE compute~$V'=\vol(\Ext(S))$ with Algorithm~\ref{algovol}
	    \RETURN $V' < 2V$
	\ENDWHILE
	\STATE Enumerate $\leftarrow$ Algorithm~\ref{enumbigball} or Algorithm~\ref{enumsmallballs}
	\STATE $S \leftarrow$ output of the Normalized Basis Algorithm~\ref{basisalg}
	\STATE $R \leftarrow$ inverse, cycle and reflection relations from Theorem~\ref{presthm}
	\STATE $G\leftarrow \gp{S | R}$
	\STATE Let~$\phi : G\to\Gamma(\order)$ be the map that evaluates words in the generators
	\STATE Let~$\psi : \Gamma(\order)\to G$ be the map that writes elements as words in the generators using Algorithm~\ref{redalg}
	\RETURN $G,\phi,\psi$
\end{algorithmic}
\end{algorithm}

\begin{rmks}\label{remdiffgroup}~
\begin{itemize}
 \item If we want to compute the group that is the image of a smaller order, or more generally a finite index subgroup~$\Gamma'$ of the group~$\Gamma(\order)$ given by a maximal order~$\order$, we can compute first a normalized basis for the larger group~$\Gamma(\order)$, and then compute the index by standard coset enumeration techniques. This gives the covolume of the smaller group, and even a set of generators for it, so we can then apply the same algorithm we described.
  \item We may also want to compute a maximal group in the commensurability class of~$\Gamma(\order)$. There are infinitely many conjugacy classes of such maximal groups, and they can be obtained as follow. Let~$\order'$ be a maximal order in~$B$, and~$S$ a finite set of primes of~$F$ that split in~$B$. Let~$\order''\subset\order'$ be an Eichler order of level~$\mathfrak{N}$ where~$\mathfrak{N}$ is the product of the primes in~$S$, and define~$\Gamma_{S,\order'}$ to be the normalizer of~$\order''$ in~$B^\times$.
Then every maximal group in the commensurability class of~$\Gamma(\order)$ is conjugate to a group~$\Gamma_{S,\order'}$ for some set~$S$ and some maximal order~$\order'$, which can be taken from a set of representatives of the conjugacy classes of maximal orders in~$B$. Note however that some of the groups~$\Gamma_{S,\order'}$ may not be maximal.
Since each of these groups is the image in~$\PSL_2(\C)$ of the normalizer of an order in~$B$, we may use the same enumeration techniques. The index is given in terms of a class group and a finite quotient of units in~$\Z_F$, which can be computed, so again we get the covolume of this larger group, and can apply the same technique.
The reader can refer to~\cite{Borel-volumes} or~\cite[Section 11.4]{mac} for details on maximal groups.
\end{itemize}
\end{rmks}

\newpage 

\section{Examples}\label{secex}

The author has implemented the algorithm described in the previous section in the computer system Magma~\cite{magma}. Our package \texttt{KleinianGroups} is available at \url{http://www.normalesup.org/~page/software.html}. Here we show some examples of the output of this code. In sections~\ref{cmpalgo} and~\ref{relprev}, the computations are performed on a $1.73$~GHz Intel i7 processor with Magma v2.18-4. The more extensive computations of sections~\ref{bigex} and~\ref{effi} are run on a $2.5$~GHz Intel Xeon E5420 processor from the PLAFRIM experimental testbed with Magma v2.17-12.

\subsection{Comparison between subalgorithms}\label{cmpalgo}
\subsubsection{Comparison between the normalized basis algorithms}\label{cmpnb}

Consider the ATR sextic field~$F$ of discriminant $-92779$ generated by an element~$t$ such that~$t^6 -t^5 -2t^4 +3t^3 -t^2 -2t +1=0$, and let~$\Z_F$ be its ring of integers. Let~$B=\quatalg{-1,-1}{F}$ be the quaternion algebra ramified only at the real places of~$F$. Let~$\order$ be a maximal order in~$B$; the choice does not matter as they are all conjugate. The Kleinian group~$\Gamma(\order)$ has covolume~$0.3007\dots$. We compare our algorithm with the naive Algorithm~\ref{naivebasisalg}. Both need a precomputation of~$3$ seconds for the computation of the coefficients of the Lobachevsky power series and~$4$ seconds for the evaluation of the Dedekind zeta function at~$2$. Our algorithm then computes a Dirichlet domain in~$2$ seconds, and enumerates~$37$ elements of~$\order$, yielding~$21$ elements of~$\Gamma(\order)$. The naive algorithm (actually we only removed the routine CheckPairing) computes the same Dirichlet domain in~$48$ seconds and has to enumerate~$16\,246$ elements of~$\order$, 
yielding~$1713$ elements of~$\Gamma(\order)
$. The fundamental domain (Figure~\ref{picdeg6}) has 18 faces and 42 edges.


\begin{figure}[ht]
\centering
\includegraphics*[width=10.9cm,keepaspectratio=true]{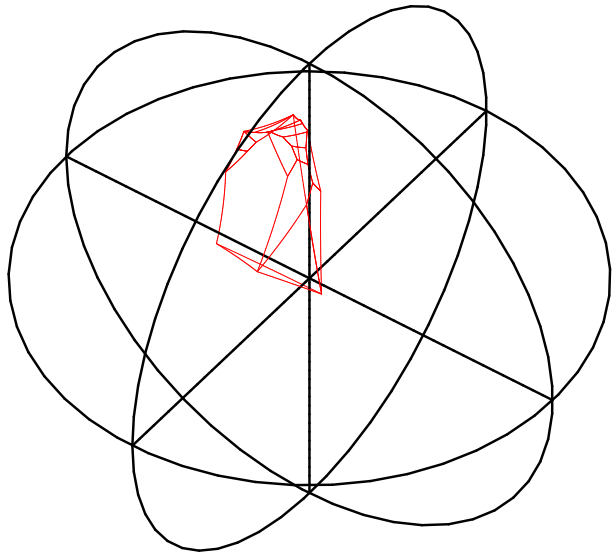}
\caption{Dirichlet domain of a Kleinian group over a sextic field}\label{picdeg6}
\end{figure}

\subsubsection{Comparison between the enumeration algorithms}\label{exenum}

Consider the ATR number field~$F$ of degree~$8$ and discriminant $-407793664$, generated by an element~$t$ such that~$t^8-4t^7+4t^6+2t^5-8t^4+4t^3+5t^2-2t-1=0$, and let~$\Z_F$ be its ring of integers. Let~$B=\quatalg{-1,-1}{F}$ be the quaternion algebra ramified only at the real places of~$F$. Let~$\order$ be a maximal order in~$B$; the choice does not matter as they are all conjugate. The Kleinian group~$\Gamma(\order)$ has covolume~$56.509\dots$. We compare the performance of our algorithm when using the enumeration algorithms~\ref{enumbigball} or~\ref{enumsmallballs}. With the deterministic enumeration algorithm~\ref{enumbigball}, our code computes a fundamental domain in~$12$~hours and~$45$~minutes ($45943$ seconds, most of which is enumeration), and enumerates~$84\,159\,799$ vectors, yielding~$1600$ group elements. With the probabilistic enumeration algorithm~\ref{enumsmallballs}, our code computes the same Dirichlet domain in~$71$~seconds, and only needs to enumerate~$3511$ vectors, yielding~$164$ 
group elements. It spends~$2$ 
seconds for computing the value of the zeta function, $16$ seconds for enumeration, $3$ seconds for the routine KeepSameGroup, $40$ for CheckPairing and $10$ for computing the volume of the polyhedron. The fundamental domain (Figure~\ref{picoctic}) has~$202$~faces and~$582$~edges.

\begin{figure}[ht]
\centering
\includegraphics*[width=10.9cm,keepaspectratio=true]{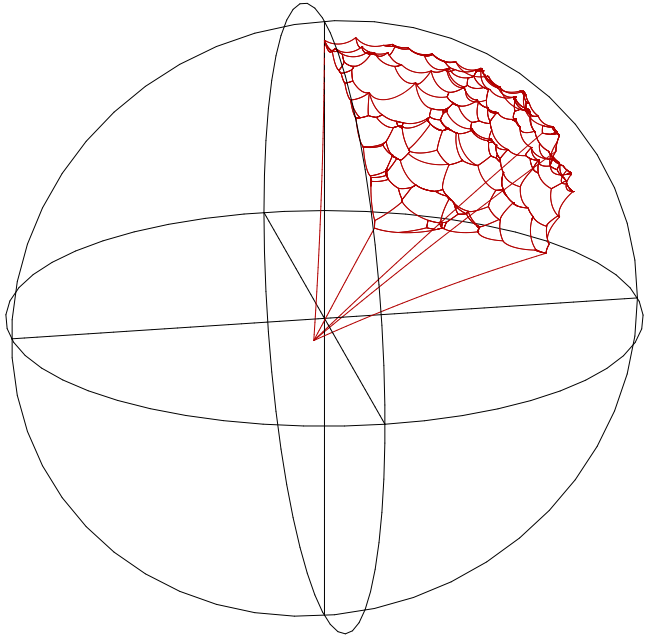}
\caption{Dirichlet domain of a Kleinian group over an octic field}\label{picoctic}
\end{figure}

\subsection{Relation to previous work}\label{relprev}
In this section we show how to recover examples covered by earlier work with our algorithm. When available, we provide a comparison of running times between public implementations and our code. The reader should keep in mind that these are only comparisons between implementations since the complexity of the algorithms is usually unknown.

\subsubsection{Bianchi groups}
Let~$F$ be an imaginary quadratic field with ring of integers~$\Z_F$. Consider the quaternion algebra~$B=\mat_2(F)$ and the maximal order~$\order=\mat_2(\Z_F)$. Then the group~$\Gamma(\order)=\PSL_2(\Z_F)$ is called a \emph{Bianchi group}. There exists already several programs computing fundamental domains for these groups~\cite{rahmhomologies, yasaki2010hyperbolic} but they only work for Bianchi groups while ours deals with general arithmetic Kleinian groups. Table~\ref{tablebianchi} gives the running time (in seconds) of our Magma package and other public implementations. 
The first three columns correspond to the discriminant of the field, its class number and the covolume of~$\PSL_2(\Z_F)$. The last four columns display running times in seconds: \texttt{Bianchi.gp}~\cite{rahmhomologies} written in GP~\cite{PARI2} implementing Swan's algorithm for~$\PSL_2(\Z_F)$, our code~\texttt{KleinianGroups} computing~$\PSL_2(\Z_F)$, the code provided by Magma implementing the algorithm of~\cite{yasaki2010hyperbolic} using Vorono\"i theory for~$\PGL_2(\Z_F)$, and our code for~$\PGL_2(\Z_F)$. Note that it is not surprising that computing~$\PGL_2(\Z_F)$ is faster: the group is larger by an index~$2$, so the covolume is twice smaller and our computation is~$4$ times shorter (see also section~\ref{effi}).


\begin{table}
\begin{tabular}{|c|c|c||c|c|c|c|}
  \hline
  $\Delta_F$ & $h_F$ & volume & \texttt{Bianchi} & \texttt{KG}, $\PSL_2$ & Magma & \texttt{KG}, $\PGL_2$\\
  \hline
  $-3$ & 1 & 0.169 & 0.015 & 0.93 & 0.43 & 0.83\\
  $-15$ & 2 & 3.139 & 0.152 & 0.92 & 0.8 & 2.32 \\
  $-23$ & 3 & 6.449 & 0.176 & 1.22 & 1.11 & 2.06\\
  $-39$ & 4 & 13.80 & 2.37 & 9.44 & 3.05 & 4.36\\
  $-47$ & 5 & 19.43 & 3.83 & 19.9 & 5.33 & 6.96\\
  $-71$ & 7 & 37.53 & 21.6 & 36.6 & 17.8 & 13.2\\
  $-87$ & 6 & 44.72 & 25.7 & 45.1 & 17.3 & 16.4\\
  $-95$ & 8 & 57.06 & 41.4 & 43.8 & 33.9 & 19.3\\
  $-119$ & 10 & 82.93 & 7080. & 137. & 99.5 & 25.6\\
  $-167$ & 11 & 132.3 & 1545. & 391. & 188. & 80.9\\
  $-199$ & 9 & 148.5 & 3840. & 393. & 224. & 92.7\\
  \hline
\end{tabular}
\caption{Running times for Bianchi groups}\label{tablebianchi}
\end{table}

\subsubsection{Arithmetic Fuchsian groups}

Let~$F$ be a totally real field and~$B$ a quaternion algebra ramified at every infinite place but one. Let~$\order$ be an order in~$B$. Then the group~$\Gamma(\order) = \order^1/\{\pm 1\}$ embeds into~$\PSL_2(\R)$, in which it is discrete with finite covolume: it is an \emph{arithmetic Fuchsian group}. Using the action of~$\PSL_2(\R)$ on the upper half-plane J.~Voight~\cite{voightfuchsian} was able to compute fundamental domains for these groups. Since we have~$\PSL_2(\R)\subset\PSL_2(\C)$, a Fuchsian group can be seen as a Kleinian group leaving a geodesic plane stable. Using this we can also compute arithmetic Fuchsian groups with our code. Our probabilistic enumeration Algorithm~\ref{enumsmallballs} leads to an improvement in high degree. As an example, consider the totally real field~$F$ with discriminant~$9685993193$, generated by an element~$t$ such that~$t^9-2t^8-7t^7+11t^6+15t^5-15t^4-10t^3+7t^2+2t-1=0$. Let~$B=\quatalg{a,b}{F}$ with~$a=-3t^8 + 2t^7 + 30t^6 - 8t^5 - 93t^4 + 90t^2 + 2t - 26$ 
and~$b=-1$. It is ramified at every real place but one. Let~$\order$ be a maximal order in~$B$. The Fuchsian group~$\Gamma(\order)$ has coarea~$103.67\dots$; our code computes a fundamental domain for this group in~$13$ minutes ($735$ seconds). The code provided by Magma and implementing the algorithm of~\cite{voightfuchsian} computes a fundamental domain for~$\Gamma(\order)$ in~$1$~hour and~$10$~minutes ($4204$ seconds).

\subsubsection{The Hamiltonians over $\Z \bigl[\frac{1+\sqrt{-7}}{2}\bigr]$}

Consider the field~$F=\Q(\sqrt{-7})$ and the quaternion division algebra~$\quatalg{-1,-1}{F}$. Then~$\order=\Z_F+\Z_F i+\Z_F j+\Z_F ij$ is a non-maximal order in~$B$. A fundamental domain for this group was computed by C.~Corrales, E.~Jespers, G.~Leal and \'A.~del~R\'io in \cite{Corrales}. Using the method of Remark~\ref{remdiffgroup}, our code can compute a fundamental domain for the group~$\Gamma(\order)$. It computes first a maximal order~$\order'\supset\order$, and a fundamental domain for~$\Gamma(\order')$ (having covolume~$0.8889\dots$). By coset enumeration, it finds that~$\Gamma(\order)$ has index~$9$ in the larger group, and computes a fundamental domain for the initial group~$\Gamma(\order)$. The overall computation takes~$15$ seconds.

\subsection{A larger example}\label{bigex}

Consider the ATR field~$F$ generated by an element~$t$ such that~$t^{10}+4t^9-18t^7-27t^6+26t^5+57t^4-2t^3-33t^2-10t+1=0$, having discriminant~$-546829505431\simeq -5.5\, 10^{11}$. Let~$B$ be the quaternion algebra~$\quatalg{a,b}{F}$ where~$a=\frac12(-25t^9 - 82t^8 + 61t^7 + 404t^6 + 376t^5 - 932t^4 - 718t^3 + 590t^2+ 368t - 33)$ and~$b=-1$. It is ramified exactly at the real places of~$F$. Let~$\order$ be a maximal order in~$B$. The group~$\Gamma(\order)$ has covolume~$1783.7\dots$. Our code computes a fundamental domain for this group in~$23$ hours and~$39$ minutes ($85150$ seconds). It spends $5.3$\% of the time for enumeration, $5.8$\% for the routine KeepSameGroup, $87.7$\% for CheckPairing and $1.3$\% for computing the volume of the polyhedron. The fundamental domain has~$5434$~faces and~$16252$~edges.

\subsection{Efficiency of the algorithm}\label{effi}

\begin{figure}[ht]
\centering
\includegraphics*[width=7cm,keepaspectratio=true]{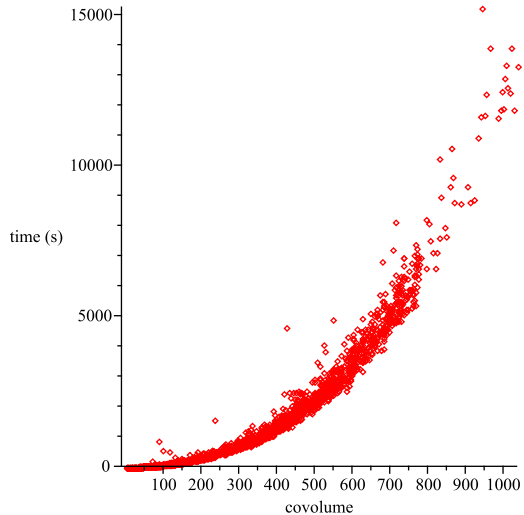}
\caption{Running time of the algorithm}\label{plottime}
\end{figure}

According to geometers, the parameter encoding the complexity of an arithmetic Kleinian group is the covolume. In practise it is simpler to vary the discriminant of the base field (and hence the degree) and the norm of the discriminant of the quaternion algebra. It seems hard to estimate the running time of the algorithm in terms of these parameters. First, we do not know any bound on the radii of the isometric spheres containing the faces of the closure of the Dirichlet domain, or of generators of the group, so we do not know how many elements we have to enumerate. Then, even if we have generators of the group, we do not know how long the normalized basis algorithm could run before terminating (see also Remark \ref{rmkterm}).

We present numerical data obtained in a family. Since the running time increases very quickly with the discriminant of the field, we fixed the base field and varied the discriminant of the algebra. The field we chose is the ATR cubic field of discriminant~$-23$. We computed groups~$\Gamma(\order)$ for every algebra with discriminant less than~$10\,000$, and one algebra every ten with discriminant less than~$15\,000$.

Analysis of this data shows that the running time is approximately proportional to the square of the covolume, with a few exceptionnally slow computations. We explain this as follows: in almost all cases, the enumeration appears to take negligible time, and the longest part is the computation of the fundamental domain itself; moreover the data (Figure \ref{plotfaces}) seem to indicate that the number of faces is proportional to the covolume (we have such a lower bound since the volume of a hyperbolic tetrahedron is bounded by~$3\Loba(\frac{\pi}{3})$), and we know that our algorithm to compute the domain given the faces is quadratic.

\begin{figure}[ht]
\centering
\includegraphics*[width=7cm,keepaspectratio=true]{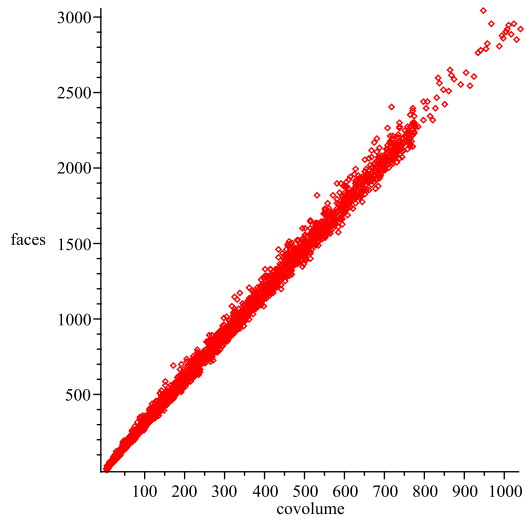}
\caption{Number of faces of the closure of the Dirichlet domains}\label{plotfaces}
\end{figure}

\bibliographystyle{alpha}
\bibliography{kln_gps}

\begin{thebibliography}{CJLdR04}

\bibitem[BCP97]{magma}
Wieb Bosma, John Cannon, and Catherine Playoust.
\newblock The {M}agma algebra system. {I}. {T}he user language.
\newblock {\em J. Symbolic Comput.}, 24(3-4):235--265, 1997.
\newblock Computational algebra and number theory (London, 1993).

\bibitem[BGLS10]{generators}
Mikhail Belolipetsky, Tsachik Gelander, Alexander Lubotzky, and Aner Shalev.
\newblock Counting arithmetic lattices and surfaces.
\newblock {\em Ann. of Math. (2)}, 172(3):2197--2221, 2010.

\bibitem[Bor81]{Borel-volumes}
A.~Borel.
\newblock Commensurability classes and volumes of hyperbolic {$3$}-manifolds.
\newblock {\em Ann. Scuola Norm. Sup. Pisa Cl. Sci. (4)}, 8(1):1--33, 1981.

\bibitem[BV13]{bvtorsion}
Nicolas Bergeron and Akshay Venkatesh.
\newblock The asymptotic growth of torsion homology for arithmetic groups.
\newblock {\em J. Inst. Math. Jussieu}, 12(2):391--447, 2013.

\bibitem[BW00]{borelwallach}
A.~Borel and N.~Wallach.
\newblock {\em Continuous cohomology, discrete subgroups, and representations
  of reductive groups}, volume~67 of {\em Mathematical Surveys and Monographs}.
\newblock American Mathematical Society, Providence, RI, second edition, 2000.

\bibitem[CFJR01]{chinburg2001arithmetic}
Ted Chinburg, Eduardo Friedman, Kerry~N. Jones, and Alan~W. Reid.
\newblock The arithmetic hyperbolic 3-manifold of smallest volume.
\newblock {\em Ann. Scuola Norm. Sup. Pisa Cl. Sci. (4)}, 30(1):1--40, 2001.

\bibitem[CJLdR04]{Corrales}
Capi Corrales, Eric Jespers, Guilherme Leal, and Angel del R{\'{\i}}o.
\newblock Presentations of the unit group of an order in a non-split quaternion
  algebra.
\newblock {\em Adv. Math.}, 186(2):498--524, 2004.

\bibitem[CV12]{venkcaltorsion}
Frank Calegary and Akshay Venkatesh.
\newblock A torsion {J}acquet--{L}anglands correspondence.
\newblock 2012.
\newblock \url{http://arxiv.org/abs/1212.3847}.

\bibitem[FP85]{fp}
U.~Fincke and M.~Pohst.
\newblock Improved methods for calculating vectors of short length in a
  lattice, including a complexity analysis.
\newblock {\em Math. Comp.}, 44(170):463--471, 1985.

\bibitem[JL70]{jacquet1972automorphic}
H.~Jacquet and R.~P. Langlands.
\newblock {\em Automorphic forms on {${\rm GL}(2)$}}.
\newblock Lecture Notes in Mathematics, Vol. 114. Springer-Verlag, Berlin,
  1970.

\bibitem[Kan83]{kannan}
Ravi Kannan.
\newblock Improved algorithms for integer programming and related lattice
  problems.
\newblock In {\em Proceedings of the fifteenth annual ACM symposium on Theory
  of computing}, STOC '83, pages 193--206, New York, NY, USA, 1983. ACM.

\bibitem[Lip02]{lip}
M.~Lipyanskiy.
\newblock A computer-assisted application of {P}oincar{\'e}'s fundamental
  polyhedron theorem.
\newblock Preprint available at
  \url{http://www.math.columbia.edu/~ums/Archive.html}, 2002.

\bibitem[Mas71]{maskit1971poincar}
Bernard Maskit.
\newblock On {P}oincar\'e's theorem for fundamental polygons.
\newblock {\em Advances in Math.}, 7:219--230, 1971.

\bibitem[MR03]{mac}
Colin Maclachlan and Alan~W. Reid.
\newblock {\em The arithmetic of hyperbolic 3-manifolds}, volume 219 of {\em
  Graduate Texts in Mathematics}.
\newblock Springer-Verlag, New York, 2003.

\bibitem[Pag10]{monmemoire}
Aurel Page.
\newblock {Computing fundamental domains for arithmetic Kleinian groups}.
\newblock Master's thesis, Universit{\'e} Paris 7, August 2010.

\bibitem[Rah10]{rahmhomologies}
Alexander Rahm.
\newblock {\em {(Co)homologies et K-th{\'e}orie de groupes de Bianchi par des
  mod{\`e}les g{\'e}om{\'e}triques calculatoires}}.
\newblock Phd thesis, Universit{\'e} Joseph-Fourier - Grenoble I, October 2010.

\bibitem[Rat06]{foundations}
John~G. Ratcliffe.
\newblock {\em Foundations of hyperbolic manifolds}, volume 149 of {\em
  Graduate Texts in Mathematics}.
\newblock Springer, New York, second edition, 2006.

\bibitem[Ril83]{riley}
Robert Riley.
\newblock Applications of a computer implementation of {P}oincar\'e's theorem
  on fundamental polyhedra.
\newblock {\em Math. Comp.}, 40(162):607--632, 1983.

\bibitem[Swa71]{swan}
Richard~G. Swan.
\newblock Generators and relations for certain special linear groups.
\newblock {\em Advances in Math.}, 6:1--77 (1971), 1971.

\bibitem[{The}11]{PARI2}
{The PARI~Group}, Bordeaux.
\newblock {\em {PARI/GP, version {\tt 2.6.0}}}, 2011.
\newblock available from \url{http://pari.math.u-bordeaux.fr/}.

\bibitem[Vig80]{mfv}
Marie-France Vign{\'e}ras.
\newblock {\em Arithm\'etique des alg\`ebres de quaternions}, volume 800 of
  {\em Lecture Notes in Mathematics}.
\newblock Springer, Berlin, 1980.

\bibitem[Voi09]{voightfuchsian}
John Voight.
\newblock Computing fundamental domains for {F}uchsian groups.
\newblock {\em J. Th\'eor. Nombres Bordeaux}, 21(2):469--491, 2009.

\bibitem[Yas10]{yasaki2010hyperbolic}
Dan Yasaki.
\newblock Hyperbolic tessellations associated to {B}ianchi groups.
\newblock In {\em Algorithmic number theory}, volume 6197 of {\em Lecture Notes
  in Comput. Sci.}, pages 385--396. Springer, Berlin, 2010.

\end{thebibliography}

\end{document}